\documentclass[11pt]{amsart}
\usepackage[]{mdframed}
\usepackage{url}
\usepackage[normalem]{ulem}
\usepackage{longtable, stackrel, multicol, fancybox, stackrel, tcolorbox, mathtools, adjustbox}
\usepackage{amsfonts, amsmath ,amssymb, mathrsfs, amsthm, lmodern} 
\usepackage{graphicx, color, eurosym, eucal}

\usepackage[marginpar=2cm]{geometry}
\usepackage{pgfplots}
\pgfplotsset{width=10cm,compat=1.9}

\normalbaselines						  %
\usepackage{xcolor, tcolorbox}
\definecolor{blue(ryb)}{rgb}{0.01, 0.28, 1.0}
\definecolor{brandeisblue}{rgb}{0.0, 0.44, 1.0}
\definecolor{ceruleanblue}{rgb}{0.16, 0.32, 0.75}
\definecolor{cobalt}{rgb}{0.0, 0.28, 0.67}
\definecolor{coolblack}{rgb}{0.0, 0.18, 0.39}
\definecolor{darkblue}{rgb}{0.0, 0.0, 0.55}
\usepackage[hidelinks]{hyperref}
\hypersetup{
    colorlinks,
    citecolor=darkblue,
    filecolor=black,
    linkcolor=darkblue,
    urlcolor=black
}

\definecolor{darkblue}{rgb}{0.2,0.2,0.6}

\definecolor{DarkBlue}{rgb}{0,0.1,0.5}

\newcommand\soutD{\bgroup\markoverwith
	{\textcolor{DarkGreen}{\rule[.5ex]{2pt}{1pt}}}\ULon}
\newcommand\soutP{\bgroup\markoverwith
	{\textcolor{blue}{\rule[.5ex]{2pt}{1pt}}}\ULon}

\newcommand\nb{\nabla}
\newcommand{\beq}{\begin{equation} \begin{split}}
\newcommand{\eeq}{\end{split} \end{equation}}

\newcommand\Omg{\Omega}

\newcommand\Sg{\Sigma}

\renewcommand\and{\qquad\text{and}\qquad}

\newcommand\sm{\setminus}

\newcommand\sfm{\mathsf{m}}

\newcommand{\comm}[1]{}

\def\sfH{\mathsf{H}}

\def\bm1{\mathbbm{1}}
\def\G{\Gamma}
\def\C{\cC}

\def\s{\sigma}

\def\p{\partial}

\def\arr{\rightarrow}

\newcommand{\ee}{\mathsf{e}}
\newcommand{\zz}{\mathsf{z}}

\renewcommand{\gg}{{\gamma}}

\def\aa{\alpha}
\def\lm{\lambda}

\def\s{\sigma}

\def\ii{{\mathsf{i}}}
\def\p{\partial}

\def\kp{\kappa}

\def\sfH{\mathsf{H}}

\def\Z{\mathsf{z}}

\def\dd{{\,\mathrm{d}}}

\newcounter{counter_a}
\newenvironment{myenum}{\begin{list}{{\rm(\roman{counter_a})}}%
{\usecounter{counter_a}
\setlength{\itemsep}{1.ex}\setlength{\topsep}{0.8ex}
\setlength{\leftmargin}{5ex}
\setlength{\labelwidth}{5ex}}}{\end{list}}

\usepackage[latin1]{inputenc}
\usepackage[T1]{fontenc}

\numberwithin{figure}{section}
\numberwithin{equation}{section}
\theoremstyle{plain}% default
\newtheorem*{thm*}{Theorem}
\newtheorem{thm}{Theorem}[section]
\newtheorem{hyp}[thm]{Hypothesis}
\newtheorem{prop}[thm]{Proposition}

\theoremstyle{remark}
\newtheorem{remark}[thm]{Remark}
\theoremstyle{plain}

\newcommand{\beu}{\begin{equation*}}
\newcommand{\eeu}{\end{equation*}}
\newcommand{\besu}{\begin{equation*}
\begin{aligned}}
\newcommand{\eesu}{\end{aligned}
\end{equation*}}
\newcommand{\bes}{\begin{equation}
\begin{aligned}}
\newcommand{\ees}{\end{aligned}
\end{equation}}

\newcommand\cA{\mathcal A}
\newcommand\cB{\mathcal B}

\newcommand\cF{\mathcal F}

\newcommand\cH{\mathcal H}
\newcommand\cK{\mathcal K}

\newcommand\cL{\mathcal L}

\newcommand\frq{\mathfrak q}

\newcommand\frp{\mathfrak p}

\newcommand\ov{\overline}
\newcommand\wt{\widetilde}

\newcommand\void[1]{}

\def\ov{\overline}

      \def\dC{{\mathbb C}}

   \def\dN{{\mathbb N}}   
      \def\dR{{\mathbb R}}

   \def\dZ{{\mathbb Z}}

   \def\sfH{{\mathsf H}}

\def\cA{{\mathcal A}}   \def\cB{{\mathcal B}}   
      \def\cF{{\mathcal F}}
   \def\cH{{\mathcal H}}   \def\cI{{\mathcal I}}
   \def\cK{{\mathcal K}}   \def\cL{{\mathcal L}}
      \def\cO{{\mathcal O}}

\def\N{\mathbb{N}}

\newcommand\Th{\Theta}
\newcommand{\Hm}[1]{\leavevmode{\marginpar{\tiny%
			$\hbox to 0mm{\hspace*{-0.5mm}$\leftarrow$\hss}%
			\vcenter{\vrule depth 0.1mm height 0.1mm width \the\marginparwidth}%
			\hbox to
			0mm{\hss$\rightarrow$\hspace*{-0.5mm}}$\\
			\relax\raggedright #1}}}

\def\nz{\ifmmode {I\hskip -3pt N} \else {\hbox {$I\hskip -3pt N$}}\fi}
\def\zz{\ifmmode {Z\hskip -4.8pt Z} \else

       {\hbox {$Z\hskip -4.8pt Z$}}\fi}
\def\qz{\ifmmode {Q\hskip -5.0pt\vrule height6.0pt depth 0pt
       \hskip 6pt} \else {\hbox
       {$Q\hskip -5.0pt\vrule height6.0pt depth 0pt\hskip 6pt$}}\fi}
\def\rz{\ifmmode {I\hskip -3pt R} \else {\hbox {$I\hskip -3pt R$}}\fi } 
\def\cz{\ifmmode {C\hskip -4.8pt\vrule height5.8pt \hskip 6.3pt} \else 
{\hbox {$C\hskip -4.8pt\vrule height5.8pt \hskip 6.3pt$}}\fi} 

\def\curl{{\rm curl}\,}% rotationnel 
\def \Ab{{\bf A}}

\def\and {{\rm \; and \;}}
\def\dist {{\rm \; dist \;}}

\def\nb{\mathbf n}
\def\R{\mathbb R}
\def\N{\mathbb N}
\def\Z{\mathbb Z}
\def\C{\mathbb C}
\def\cF{\mathcal F}
\definecolor{ao(english)}{rgb}{0.0, 0.5, 0.0}

\newcommand{\wg}{\widetilde\gamma}
\newcommand{\wG}{\widetilde\Gamma}

\renewcommand*{\overrightarrow}[1]{\vbox{\halign{##\cr 
  \tiny\rightarrowfill\cr\noalign{\nointerlineskip\vskip1pt} 
  $#1\mskip2mu$\cr}}}
\makeatletter
\@addtoreset{equation}{section}

\makeatother

\newtheorem{theorem}{Theorem}[section]
\newtheorem{conjecture}[theorem]{Conjecture}
\newtheorem{lemma}[theorem]{Lemma}
\newtheorem{proposition}[theorem]{Proposition}
\newtheorem{definition}[theorem]{Definition}
\newtheorem{corollary}[theorem]{Corollary}

\newcommand{\dx}{\,\mathrm{d}}

\newcommand{\dom}{\mathrm{dom}\,}
\def\p{\partial}

\def\Omg{\Omega}
\def\dC{{\mathbb C}}
\def\sfH{{\mathsf H}}
\def\ii{{\mathrm{i}}}

\title[Isoperimetric inequalities for inner parallel curves]{Isoperimetric inequalities for inner parallel curves}
\author[C. Dietze]{Charlotte Dietze}
\address{Department of Mathematics, LMU Munich, Theresienstr.~39, 80333 Munich, Germany}
\email{dietze@math.lmu.de}
\author[A. Kachmar]{Ayman Kachmar}
\address{The Chinese University of Hong Kong (Shenzhen), School of Science and Engineering, Longgang District, Shenzhen, China.}
\email{akachmar@cuhk.edu.cn}
\author[V. Lotoreichik]{Vladimir Lotoreichik}
\address{Department of Theoretical Physics, Nuclear Physics Institute, Czech Academy of Sciences, 25068 
\v{R}e\v{z}, Czech Republic}
\email{lotoreichik@ujf.cas.cz}


\begin{document}
\bibliographystyle{plain}
\maketitle 
%address 

%
\begin{abstract}
We prove weighted isoperimetric inequalities for smooth, bounded, and simply connected domains. More precisely, we show that the moment of inertia of inner parallel curves for domains with fixed perimeter attains its maximum for a disk. This inequality, which was previously only known for convex domains, allows us to extend an isoperimetric inequality for the magnetic Robin Laplacian to non-convex centrally symmetric domains. 
Furthermore, 
we extend our isoperimetric inequality for moments of inertia, which are second moments, to $p$-th moments for all $p$ smaller than or equal to two. We also show that the disk is a strict local maximiser in the nearly circular, centrally symmetric case for all $p$ strictly less than three, and that the inequality fails for all $p$ strictly bigger than three. 
\end{abstract}

\tableofcontents

\section{Introduction}\label{sec:int}
Let $\Omega \subset \mathbb{R}^{2}$ be a smooth, bounded and simply connected domain. For any $t\ge0$, we define the corresponding inner parallel curve $S_{t}$ by
\begin{align}\label{eq:stdef}
S_{t}:=S_t(\Omg)=\{x \in \overline{\Omega} \mid \operatorname{dist}(x, \partial \Omega)=t\} \text {. }
\end{align}
The systematic study of the geometric structure and regularity of inner parallel curves was initiated in~\cite{B41, F41, H64}, see also~\cite{SST,S01} and references therein.

By~\cite[Theorem 4.4.1]{SST} and~\cite[Proposition A.1]{S01}, the inner parallel curve $S_{t}$ is a finite union of piecewise smooth simple curves for almost every $t\ge0$. Hartman~\cite[Corollary 6.1]{H64} showed that
\begin{align}\label{eq:sthartest}
\left|S_{t}\right| \leq|\partial \Omega|-2 \pi t \quad \text { for almost every } t\ge0 \text { with $S_t \ne \varnothing$, }
\end{align}
where $\left|S_{t}\right|$ and $|\partial \Omega|$ denote the length of $S_{t}$ and $\partial \Omega$, respectively.

The moment of inertia of the inner parallel curve $S_{t}$ computed with respect to its centroid 
\begin{equation}\label{eq:defcentS_t}
    c(t):=\frac{1}{|S_t|}\int_{S_{t}} x \dd \mathcal{H}^{1}(x)
\end{equation}
is given by
\begin{align}\label{eq:momindef}
\int_{S_{t}}|x-c(t)|^{2} \dd \mathcal{H}^{1}(x)\,,
\end{align}
where $\mathcal{H}^{1}$ denotes the one-dimensional Hausdorff measure. 
In this paper, we address the following question: 
\begin{center}
\emph{Fixing the perimeter of $\Omega$, what shape of $\Omega$ maximizes the moment of inertia of the inner parallel curve $S_{t}$ as defined in \eqref{eq:momindef} for given $t\ge0$? }
\end{center}
Our main result states that the optimal shape is attained for a disk.

\begin{theorem}[An isoperimetric inequality for moments of inertia]\label{th:isomom}
Let $\Omega \subset \mathbb{R}^{2}$ be a smooth, bounded and simply connected domain. Then, for almost every $t\ge0$, 
\begin{align}\label{eq:isomom}
\int_{S_{t}(\Omg)}|x-c(t)|^{2} \dd\mathcal{H}^{1}(x) \leq \int_{S_{t}(\cB)}|x|^{2} 
\dd\mathcal{H}^{1}(x),
\end{align}
where 
$\cB$ is the disk centered at the origin and with the same perimeter as $\Omg$.
Here $S_{t}(\cdot)$ and $c(t)$ are defined in \eqref{eq:stdef} and \eqref{eq:defcentS_t}, respectively. When $t\in\left[0,\frac{|\partial\Omega|}{2\pi}\right)$, the equality in \eqref{eq:isomom} is attained if and only if $\Omega$ is a disk.
\end{theorem}

Note that $c(t)=0$ if $\Omega$ is a disk centred at the origin. In the setting where $S_{t}$ is a closed curve, the statement of Theorem \ref{th:isomom} can be deduced from a result due to Hurwitz~\cite[pp.~396-397]{Hu} combined with \eqref{eq:sthartest}. For instance, this argument applies for convex domains $\Omega$, see~\cite[p.~12]{KL} for further details. In general, $S_{t}$ can consist of several connected components, see for example Figure~\ref{fig-dumbbell},  and Theorem \ref{th:isomom} is novel in this case.
The classical result by Hurwitz itself provides an isoperimetric inequality for the moment of inertia of the boundary of a planar domain under fixed perimeter constraint and essentially coincides with the statement of Theorem~\ref{th:isomom} in the special case $t=0$. The recent contribution~\cite{LS23} proves
a quantitative version of the inequality by Hurwitz and addresses the higher dimensional setting.   

Our proof of Theorem~\ref{th:isomom} relies on an explicit construction of a \emph{closed} curve $\Sigma_{t}$.
An illustration for $\Sigma_{t}$ is shown in Figure~\ref{fig-dumbbell2} (where $\Sigma_t$ inherits the symmetry of $\Omega$).
\begin{figure}[htp]
  \centering
  \begin{minipage}[t]{0.49\textwidth}
    \includegraphics[width=\textwidth]{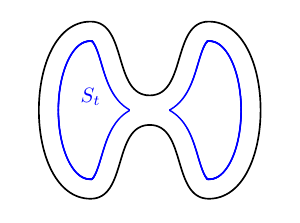}
    \caption{A schematic representation of the inner parallel curve $S_t$ in a dumbbell-like domain. Note that $S_t$ can be disconnected.}
    \label{fig-dumbbell}
  \end{minipage}
  \hfill
  \begin{minipage}[t]{0.49\textwidth}
    \includegraphics[width=\textwidth]{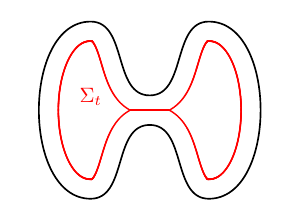}
    \caption{A schematic illustration of the connected curve $\Sigma_t$ in the case of a dumbbell-like domain. The proof of Theorem~\ref{th:curveconstr} will show that the segment connecting the two connected components of $S_t$ is doubly covered.}
\label{fig-dumbbell2}
  \end{minipage}
\end{figure}
\begin{theorem}[Covering inner parallel curves with a closed curve]\label{th:curveconstr}
Let $\Omega \subset \mathbb{R}^{2}$ be a smooth, bounded and simply connected domain. Then, for almost every $t\ge0$ with $S_{t} \neq \varnothing$ there exists a closed and piecewise smooth curve $\Sigma_{t}$ with
\begin{align}
S_{t} \subset \Sigma_{t} \quad \text { and } \quad\left|\Sigma_{t}\right| \leq|\partial \Omega|-2 \pi t,
\end{align}
where $S_{t}$ was defined in \eqref{eq:stdef}.
\end{theorem}
The technical result Theorem~\ref{th:curveconstr} is of independent interest as we obtain an improved version of \eqref{eq:sthartest} taking the distance between different connected components of $S_{t}$ into account, see Corollary \ref{co:hartimp} below.

\bigskip

More generally, we can consider $p$-th moments and ask for which $p \in(0, \infty)$ we have

\begin{align}\label{eq:stp}
\int_{S_{t}(\Omega)}|x-c(t)|^{p} \dd \mathcal{H}^{1}(x) \leq \int_{S_{t}(\cB)}|x|^{p} \dd \mathcal{H}^{1}(x) \quad \text { for almost every } t \geq 0 \text {, }
\end{align}
where $c(t)$ is the centroid of $S_{t}(\Omega)$ and $\cB$ is a disk centred at the origin with $|\partial \Omega|=|\partial \cB|$. For $t\ge0$ small enough, we have $|S_{t}(\Omega)|=|S_{t}(\cB)|$ and $S_{t}(\Omega)$ is a closed curve, see Lemma~\ref{lem:alphat}\,(ii) below, so \eqref{eq:stp} reduces to 
\begin{equation}
    \frac{(2 \pi)^{p}}{\left|{S_{t}(\Omega)}\right|^{p+1}} \int_{S_{t}(\Omega)}|x-c(t)|^{p} \dd \mathcal{H}^{1}(x) \leq 1.
\end{equation}
So we may ask if we have 
\begin{align}\label{eq:gammacponeintro}
\frac{(2 \pi)^{p}}{\left|\Gamma\right|^{p+1}} \int_{\Gamma}|x|^{p} \dd \mathcal{H}^{1}(x) \leq 1
\end{align}
for all closed Lipschitz curves $\Gamma$ with the origin as its centroid. 

\medskip

Note that the centroid $c(t)$ is independent of $t$ for all centrally symmetric domains $\Omega$, or for example for domains with two not necessarily orthogonal axes of symmetry. To keep things simple, we focus on the centrally symmetric case. 

\begin{theorem}[An isoperimetric inequality for $p$-th moments]\label{th:pthmoments}
    \begin{enumerate}
        \item[(i)] The statement of Theorem \ref{th:isomom} extends to \eqref{eq:stp} in the case $p\in(0,2]$.
        \item[(ii)] For $p<3$, the boundary of a disk is a strict local maximiser among nearly circular, centrally symmetric  closed Lipschitz curves $\Gamma$ of the left hand side in the inequality \eqref{eq:gammacponeintro}.
\item[(iii)] For $p>3$, \eqref{eq:gammacponeintro} does not hold, not even locally near boundary of the disk. More precisely, there exists a sequence of nearly circular, centrally symmetric closed Lipschitz curves 
$\left(\Gamma_n\right)_{n\in\N}$ converging uniformly to the boundary of the disk for which
\begin{align*}
\frac{(2 \pi)^{p}}{\left|\Gamma_{n}\right|^{p+1}} \int_{\Gamma_{n}}|x|^{p} \dd \mathcal{H}^{1}(x)>1.
\end{align*}
    \end{enumerate}
\end{theorem}

This naturally leads to the following conjecture.
\begin{conjecture}\label{con:p3}
\eqref{eq:gammacponeintro} holds for all $p\leq 3$ and all closed Lipschitz curves 
$\Gamma$ with the origin as its centroid.
\end{conjecture}

In the case $p \in(0,2]$, Theorem \ref{th:pthmoments} follows from Theorem \ref{th:isomom} and Theorem \ref{th:curveconstr} using Jensen's inequality. For the local optimality for $p<3$ in Theorem \ref{th:pthmoments} (ii), we follow a Fuglede-type argument \cite{Fu}. From these computations, we also obtain Theorem \ref{th:pthmoments} (iii), where symmetry breaking occurs for $p>3$.

\bigskip

Theorem \ref{th:isomom} and  Theorem \ref{th:pthmoments} are of general interest as (weighted) isoperimetric inequalities have recently received great attention \cite{Al, BBMP99, CROS, CGPRS}, see also \cite{Fusco, Ossiso, OssBon, Bon, CL} on quantitative isoperimetric inequalities. 
In the present paper, we consider the moment of inertia of the inner parallel curves $S_{t}$ and compare it with the corresponding quantity for
a disk of the same perimeter. This is a relatively unusual setting as our constraints do not involve the area of the domain $\Omega$, but only its perimeter. 
In the case of centrally symmetric domains, or more generally for domains $\Omega$ for which the centroid $c(t)$ defined in \eqref{eq:defcentS_t} is independent of $t$, we can deduce from Theorem \ref{th:isomom}  a result going back to Hadwiger~\cite{H56}, see Corollary \ref{co:integriso} by integrating over $t$.

\bigskip

As an application of Theorem \ref{th:isomom}, we obtain an isoperimetric inequality for the magnetic Robin Laplacian. More precisely, considering the magnetic Robin Laplacian with a negative boundary parameter $\beta$ and a sufficiently small constant magnetic field $b$, the ground state energy is expressed as follows
\[
\lambda_{1}^{\beta,b}(\Omega)=
\inf_{\begin{smallmatrix}
u\in H^1(\Omg)\\
\|u\|_{L^2(\Omg)}=1    
\end{smallmatrix} }
\left(\int_{\Omg}\left|(-\ii\nabla  -b{\bf A})
u\right|^2+\beta\int_{\p\Omg}|u|^2\dd\cH^1(x)\right),\]
where ${\bf A}$ is a vector field in $\Omg$ with $\curl {\bf A} = 1$.
It was shown in~\cite[Theorem 4.8]{KL} that the corresponding ground state energies for convex and centrally symmetric domains $\Omega$ and a disk $\cB$ of the same perimeter satisfy $\lambda_{1}^{\beta,b}(\Omega) \leq \lambda_{1}^{\beta,b}(\cB)$. 
Using Theorem \ref{th:isomom}, we can remove the convexity assumption on $\Omega$. 
\begin{theorem}[An isoperimetric inequality for the magnetic Robin Laplacian]\label{th:impmagintro}
    Let $\Omega \subset \mathbb{R}^{2}$ be a smooth, bounded and simply connected domain. Assume that $\Omg$ is centrally symmetric or, more generally, that the centroid of $S_t(\Omg)$ is  independent of $t$
    for all $t\ge0$ with $S_t(\Omg)\neq\emptyset$. Let $\beta < 0$ be the negative Robin parameter, and let $0<b<b_0(|\partial\Omega|, \beta)$, where $b_0(|\partial\Omega|, \beta)$ depends on $|\partial\Omega|$ and $\beta$. Then the lowest eigenvalue of the magnetic Robin Laplacian on $\Omega$ with constant magnetic field of strength $b$ and Robin boundary conditions with parameter $\beta$ satisfies
    \begin{equation}\label{eq:magintro}
        \lm_1^{\beta,b}(\Omg) \le \lm_1^{\beta,b}(\cB)\,,
    \end{equation}
    where  $\cB\subset\dR^2$ is the disk having the same perimeter as $\Omg$. Equality in \eqref{eq:magintro} occurs if and only if $\Omg$ is a disk.
\end{theorem}
If $|\p\Omg|=2\pi$, then we have the explicit expression $b_0(|\p\Omg|,\beta)=\min\big(1,4\sqrt{-\beta})$.

\bigskip

{\bf Structure of the paper. }
The rest of the paper is organized as follows.  In Section~\ref{sec:pre},  we introduce some notation and auxiliary results on inner parallel curves.   In Section~\ref{sec:proof},  we prove Theorem \ref{th:curveconstr}. We use this in Section \ref{sec:proofmomin} to prove Theorem~\ref{th:isomom}. In Section \ref{sec:proofpmom}, we show Theorem~\ref{th:pthmoments}. More precisely, the proof and the precise statement of Theorem~\ref{th:pthmoments} (i) can be found in Corollary~\ref{co:iso-p-m} and for Theorem~\ref{th:pthmoments} (ii), (iii) we refer to Proposition~\ref{prop:p-moment}. Some background material on the magnetic Robin Laplacian and the proof of Theorem~\ref{th:impmagintro} is given in Section \ref{sec:magnetic}. In Section \ref{sec:appl}, we present two simple applications of Theorem~\ref{th:isomom} and  Theorem~\ref{th:curveconstr}, namely Corollary~\ref{co:hartimp} on an improved version of \eqref{eq:sthartest}, and Corollary~\ref{co:integriso} on moments of inertia of domains. 
Appendix~\ref{app} contains an elementary proof of Mirek Ol\v{s}\'{a}k on the auxiliary geometric inequality given in Proposition~\ref{prop:geometric_bound}.

\bigskip

\section{Preliminaries}\label{sec:pre}

\subsection{Notation}\label{ssec:notation}
We introduce for a piecewise-$C^1$ mapping $\gamma\colon [0,L]\arr\dR^2$ the length of the
closed and not necessarily simple curve parametrized by $\gamma$,
\[
	\ell(\gamma) := \int_0^L |\gamma'(s)|\dd s.
\]
We also use the notation $\gamma([a,b]) = \{\gamma(s)\colon s\in[a,b]\}$ for $a,b\in[0,L]$, $a < b$. We say that $\gamma_1,\gamma_2\colon[0,L]\arr\dR^2$ parametrize the same curve if there exists a continuous bijection $\psi\colon [0,L]\arr [0,L]$ such that $\gamma_1 = \gamma_2\circ\psi$.
A subset $U\subset\dR^2$ is said to be \emph{centrally symmetric} if it coincides
with its reflection $\{-x\colon x\in U\}$ with respect to the origin.
\subsection{Inner parallel curves} \label{ssec:inner}
Let $\Omega\subset\dR^2$ be a bounded, simply-connected smooth domain.
In this subsection we recall some properties of the inner parallel curves of $\Omg$. 

\subsubsection*{Parametrization of the boundary.}

Let us denote by $L=|\p\Omg|$ the perimeter of $\Omg$. Consider the arc-length parametrization of $\p\Omg$ oriented in the  counter-clockwise direction,
\[
	s\in\R/ (L\Z) \mapsto \gamma(s)=\big(\gamma_1(s),\gamma_2(s) \big)^\top\in\R^2,
\]
which identifies $\p\Omega$ with $\R/ (L\Z) \simeq [0,L)$; the function $\gamma$ is smooth which matches with the smoothness hypothesis we imposed on $\p\Omg$.

The vector $\gamma'(s)=\big(\gamma'_1(s),\gamma'_2(s) \big)^\top$ is the unit tangent vector to $\p\Omg$ at $\gamma(s)$ and points in the counter-clockwise direction. The unit normal vector at $\gamma(s)$ pointing inwards the domain $\Omg$ is given by
\begin{equation}\label{eq:def-n}
\nb(s)=(-\gamma'_2(s), \gamma'_1(s))^\top.
 \end{equation}
We introduce the curvature 
\begin{equation}\label{eq:curvature}
	\kp(s) :=\gamma_2''(s)\gamma'_1(s)- \gamma_1''(s)\gamma'_2(s)
\end{equation}
of $\p\Omg$ at the point $\gamma(s)$.
In particular, the Frenet formula
\begin{equation}\label{eq:def-k}
	\gamma''(s)=\kp(s)\nb(s)\,,
\end{equation}
holds. Recall that, since $\p\Omg$ is a smooth closed simple curve,  the total curvature identity
\cite[Corollary 2.2.2]{Kl78} yields
\begin{equation}\label{eq:tot-curv}
\int_0^{L}\kp(s)\dx s=2\pi\,.
\end{equation}
We remark that within the chosen sign convention the curvature of a convex domain is non-negative.

\subsubsection*{Properties of inner parallel curves}

We define the \emph{in-radius} of $\Omg$ by
\begin{equation}\label{eq:inradius}	
	r_{\rm i}(\Omg) := \max_{x\in\Omg}\rho(x),
\end{equation}
where $\rho$ is the distance function given by
\begin{equation}\label{eq:distance}
	\rho(x)\colon\Omg\arr\dR_+,\qquad \rho(x) := \inf_{y\in\p\Omg} |x-y|.
\end{equation}
Recall that the inner parallel curve for $\Omg$ is 
the level set of the distance function
\begin{equation}\label{eq:stdef*}
	S_t = \{x\in\ov\Omg\colon\rho(x) = t\},\qquad t\in [0,r_{\rm i}(\Omg)).
\end{equation}
For almost every $t\in(0,r_i(\Omg))$,  
the inner parallel curve $S_t$ is a finite union of disjoint piecewise smooth simple closed curves,
%.
and the curve $S_t$ admits a parametrization as in Proposition~\ref{prop:parallelcurve} below,
which was proved in ~\cite{H64}, see also  
\cite[Theorem 4.4.1]{SST} and \cite[Propositoin A.1]{S01} for more modern presentations and further refinements.
\begin{prop}\label{prop:parallelcurve}
	There exists a subset $\cL\subset
	[0,r_{\rm i}(\Omg))$, whose complement is of Lebesgue measure zero, such that for any $t\in\cL$,
	there exist $m\in\dN$ and
	\[
		0\le a_1 < b_1 < a_2 < b_2 < \dots a_m < b_m\le L,
	\]	 
	such that the inner parallel curve $S_t$ consists of
	the union of finitely many smooth curves parametrized by
	\[
		[a_k,b_k]\ni s\mapsto \gamma(s) +t{\bf n}(s),\qquad k\in\{1,2,\dots,m\},
	\]
	which forms a union of finitely many  piecewise-smooth simple closed curves.
\end{prop}
Consider the mapping
\begin{equation}\label{eq:mapping}
(s,t)\in  \dR/(L\dZ)\times\big(0,r_{\rm i}(\Omg)\big)\mapsto \Phi(s,t):=\gamma(s)+t\nb(s)\in\R^2
\end{equation}
According to~\cite[Theorem 5.25]{L18} there exists $t_\star \in (0,r_{\rm i}(\Omg))$ such that the restriction of the mapping $\Phi$ in~\eqref{eq:mapping}
to the set $\dR/(L\dZ)\times(0,t_\star)$ is a smooth diffeomorphism onto its range. The range of this restriction is then given by a tubular neighbourhood of $\p\Omg$. It is not difficult to verify that 
$S_t = \Phi(\dR/(L\dZ),t)$ for all $t\in (0,t_\star)$. However, for $t \ge t_\star$ the same property, in general, does not hold. It can also be easily checked that for all $t\in(0,t_\star)$, the inner parallel curve $S_t$ is connected and $|S_t| = L-2\pi t$.

\begin{lemma}\label{lem:curv_bound}
	Let $t\in\cL$ and the associated numbers
	$m\in\dN$, $\{a_k\}_{k=1}^m$, $\{b_k\}_{k=1}^m$, be as in Proposition~\ref{prop:parallelcurve}.
	Then, for any $k\in\{1,2,\dots,m\}$ and any $s_0\in[a_k,b_k]$, it holds that $\kp(s_0)\le \frac{1}{t}$.
\end{lemma}
\begin{proof}
	Let us introduce the notation
	\[
		\sfm(t):=\gamma(s_0)+t\nb(s_0).
	\]
	By the Frenet formula \eqref{eq:def-k} one has  $\gamma''(s_0)=\kp(s_0)\nb(s_0)$ and by Taylor's formula near $s_0$ we get
	\[
	\gamma(s)=\gamma(s_0)+(s-s_0)\gamma'(s_0)+\frac12(s-s_0)^2\kp(s_0)\nb(s_0)+\cO(|s-s_0|^3),\qquad s\arr s_0.
	\]
	Consequently, using orthogonality of $\gamma'(s_0)$ and $\nb(s_0)$ we get
	\[{\rm dist}(\gamma(s),\sfm(t))^2= t^2+(1-t\kp(s_0))(s-s_0)^2+\cO(|s-s_0|^3),\qquad s\arr s_0.\]
	Since $\sfm(t)\in S_t$,  then  ${\rm dist}(\gamma(s),\sfm(t))\geq t$ for $s$ in a neighbourhood of $s_0$, which is possible only when $1-t\kp(s_0)\geq 0$.
\end{proof}
In the next lemma we provide a simple construction of a closed but not necessarily simple curve which contains $S_t$.
The geometric bound as in Theorem~\ref{th:curveconstr} on its length
will only hold for $t$ {not larger than the inverse of 
the maximum of the curvature for the curve $\gamma$:
\[
	\kp_{\rm max}(\Omg) := \max_{s\in\dR/(L\dZ)}\kp(s)
\]

\begin{lemma}\label{lem:alphat}
	For $t\in(0,r_{\rm i}(\Omg))$ the mapping
	\begin{equation}\label{eq:aat}
	s\in\R/(L\Z)\mapsto \alpha_t(s):=\gamma(s)+t\nb(s)\in\R^2
	\end{equation}
	parametrizes a smooth closed, not necessarily simple curve
	such that
	\begin{myenum}
		\item $S_t\subset\aa_t([0,L])$
		for all $t\in\cL$.
		\item 
		$\ell(\aa_t) = L-2\pi t$ for all $t \le \frac{1}{\kp_{\rm max}(\Omg)}$,
		\item 
		$\ell(\aa_t) > L-2\pi t$
		for all $\frac{1}{\kp_{\rm max}(\Omg)} < t < r_{\rm i}(\Omg)$.
	\end{myenum}
\end{lemma}
\begin{remark}
	According to~\cite{PI} the domain $\Omg$ contains a disk of radius $\frac{1}{\kp_{\rm max}(\Omg)}$. In other words, it holds that
	\begin{equation}\label{eq:PI}
	r_{\rm i}(\Omg) \ge \frac{1}{\kp_{\rm max}(\Omg)}.
	\end{equation}
	The equality occurs for some special types of domains such as a disk or (going beyond smooth domains) for a convex hull of two disjoint disks of equal radius. The original work~\cite{PI} is hardly available and the complete proof can be found in~\cite[Proposition 2.1]{HT95}.
    The inequality~\eqref{eq:PI} shows that, in general, a more sophisticated method than in Lemma~\ref{lem:alphat} is needed 
    to construct for any $t\in (0,r_{\rm i}(\Omg))$ a closed curve of length not larger than $L-2\pi t$, which contains the inner parallel curve $S_t$.
\end{remark}

\begin{proof}[Proof of Lemma~\ref{lem:alphat}]
	Notice that smoothness  of $\gamma$ on $\dR/(L\dZ)$ ensures that $\alpha_t$ is smooth on $\dR/(L\dZ)$ as well.
	It is also clear from Proposition~\ref{prop:parallelcurve}
	that $S_t\subset\aa_t([0,L])$ for all $t\in\cL$.
	
	Using the identity~\eqref{eq:tot-curv} 
	we get for any $t \le \frac{1}{\kp_{\rm max}(\Omg)}$
	\[
	\ell(\aa_t)=\int_0^L|\dot\alpha(s)|\dx s = \int_0^L|1-t\kp(s)|\dd s=\int_0^L(1-t\kp(s))\dd s=L-2\pi t.
	\]
	Analogously we get for any $t \in(\frac{1}{\kp_{\rm max}(\Omg)},r_{\rm i}(\Omg))$
	\[
		\ell(\aa_t)=\int_0^L|\dot\alpha(s)|\dx s = \int_0^L|1-t\kp(s)|\dd s> \int_0^L(1-t\kp(s))\dd s=L-2\pi t.\qedhere
	\]
\end{proof}

In the remainder of this subsection we will discuss the properties of $S_t$ for a centrally symmetric domain $\Omg$. The central symmetry of $\Omg$ is inherited by $\p\Omg$. Consequently, if $y=\gamma(s)\in\p\Omega$, we know that $-y\in\p\Omega$ too;  moreover the centroid of $\p\Omg$ is the origin, so
\begin{equation}\label{eq:centroid-gam}
\int_0^L\gamma(s)\dx s=0\,.
\end{equation}
\begin{lemma}\label{lem:symmetry}
Let $\Omg\subset\dR^2$ be a bounded, simply-connected, centrally symmetric smooth domain. Then, for all $t\in \big(0,r_{\rm i}(\Omg)\big)$, the
inner parallel curve $S_t\subset\Omg$
is centrally symmetric.  
\end{lemma}
\begin{proof}
Let $x\in S_t$ be fixed. Then
$\rho(x) = t$ and there exists a point $y\in\p\Omg$ such that $|x-y| = t$. 
Observe now that $\rho(-x) \le t$, because $-y\in\p\Omg$. In the case that $\rho(-x) < t$ there would exist a point $z\in\p\Omg$ such that $|z+x| < t$. Since $-z\in\p\Omg$, we would get that $\rho(x) < t$, leading to a contradiction. Thus, we infer that $\rho(-x) =t$ and hence $-x\in S_t$.
\end{proof}

\subsection{An auxiliary geometric inequality}\label{sec:agi}
\label{sec:geom_ineq}
The aim of this subsection is to provide a geometric inequality, which will be used in the proof of Theorem~\ref{th:curveconstr}. 
\begin{hyp}\label{hyp}
	Let $c_1,c_2\in\dR^2$ and $t > 0$ be fixed. Let a smooth simple non-closed curve $\Gamma\subset\dR^2$ be parametrized by the arc-length via the mapping $\gg\colon [s_1,s_2]\arr\dR^2$, $s_1 < s_2$.
	Assume that the following properties hold. 
	\begin{myenum} 
		\item $p_j:=\gamma(s_j) \in \p\cB_t(c_j)$ for $j=1,2$.
		\item $\gg'(s_j)$ is tangent to $\p\cB_t(c_j)$ in the counterclockwise direction for $j=1,2$.
		\item $\G$ can be extended up to a closed simple curve so that $\cB_t(c_1)\cup\cB_t(c_2)$ is surrounded by this extension.
	\end{myenum}	
\end{hyp}
\begin{prop}\label{prop:geometric_bound}
	Under Hypothesis~\ref{hyp} the following geometric inequality holds 
	\begin{equation}\label{eq:ineq}
		|\Gamma| \ge |c_1-c_2| +t\int_{s_1}^{s_2} \kp(s) \dd s,
	\end{equation}
	where $\kp$ is the curvature of $\G$ 
	defined as in~\eqref{eq:curvature}.
\end{prop}
\begin{remark}
A possible way to prove Proposition \ref{prop:geometric_bound} is to use an abstract result due to Chillingworth~\cite[Theorem 3.3]{C72}, which states that two closed homotopic curves are regularly homotopic if the curves are direct, that is, the corresponding curves in the covering space are simple. Let us sketch the proof idea. 
\medskip

The curve $\Gamma$ is homotopic to its projection $\Sigma$ on the convex hull of the two disks. After modifying $\Sigma$ suitably so we avoid nullhomotopic loops, we can extend $\Gamma$ and $\Sigma$ to closed direct curves $\tilde\Gamma$, $\tilde\Sigma$ using the same extension. By~\cite[Theorem 3.3]{C72}, $\tilde\Gamma$ and $\tilde\Sigma$ are regularly homotopic, so the integral over their curvatures agree: $\int_{\tilde\Gamma}\kappa=\int_{\tilde\Sigma}\kappa$, see for example~\cite{W37}. Since $\Gamma$ and $\Sigma$ were extended using the same extension, we also get $\int_{\Gamma}\kappa=\int_{\Sigma}\kappa$. Finally, one can check that $\Sigma$ satisfies \eqref{eq:ineq} and by $|\Gamma| \ge|\Sigma|$, we obtain \eqref{eq:ineq} for $\Gamma$.
\end{remark}
We provide in Appendix~\ref{app} a more elementary proof of Proposition~\ref{prop:geometric_bound}, which was suggested by Mirek Ol\v{s}\'{a}k.
 \section{Proof of Theorem~\ref{th:curveconstr} -- Covering inner parallel curves}\label{sec:proof}

The aim of this section is to prove the following theorem,  which yields Theorem~\ref{th:curveconstr}.

\begin{theorem}\label{thm:main}
	There exists a subset $\cL\subset[0,r_{\rm i}(\Omg))$ such that $(0,r_{\rm i}(\Omg))\sm\cL$ is of Lebesgue measure zero, and for any $t\in \cL$, there exists a
	piecewise smooth continuous
	mapping $\s_t\colon\dR/(L\dZ)\arr\dR^2$ such that
	\begin{myenum}
		\item $S_t \subset \s_t([0,L])$.
		\item $\ell(\s_t) \le L-2\pi t$.
		\item
		For centrally symmetric domains $\Omg$,  the curve  $\s_t([0,L])$ is centrally symmetric too.
	\end{myenum}
\end{theorem} 

In the following,  let  the set \(\cL\) be  as in  Proposition~\ref{prop:parallelcurve}.  Let $t\in\cL$ be fixed.  Then,  there exist $m\in\dN$ and
\[
	0\le a_1 < b_1 < a_2 < b_2 <\dots < a_m < b_m\le L 
\]
such that the inner parallel curve $S_t$ is given by
\[
	S_t = \bigcup_{k=1}^m \big\{\gamma(s) +t {\bf n}(s)\colon s\in[a_k,b_k]\big\}.
\]
Without loss of generality, we can always reparametrize the boundary of $\Omg$ so that $a_1 = 0$. In the following we assume that such a re-parametrization is performed and that $a_1 = 0$. Moreover, for the sake of convenience we also set $a_{m+1} := L$ and $b_{m+1} := b_1$.

The inner parallel curve $S_t$, $t\in\cL$, is not necessarily connected and, in general, it consists of finitely many piecewise-smooth, simple, closed curves.  Note that, while \(S_t\) consists of finitely many piecewise smooth simple closed curves,  the pieces \(C_k:=\{\gamma(s)+t\nb(s)\colon s\in[a_k,b_k]\}\) are not necessarily   closed curves. Our aim is to construct a piecewise smooth, closed, not necessarily simple curve, which contains $S_t$ and whose length is not larger than $L-2\pi t$. The idea is to connect the terminal point of $\{\gg(s) + t\nb(s)\colon s\in[a_k,b_k]\}$ with the starting point of $\{\gg(s) + t\nb(s)\colon s\in[a_{k+1},b_{k+1}]\}$ for all $k\in\{1,2,\dots,m\}$. 

Using the computation in the proof of Lemma~\ref{lem:alphat},  and that   \(\kappa(s)\leq \frac1t\) for all \(s\in[a_k,b_k]\) by  Lemma~\ref{lem:curv_bound},  we get the following expression for the length of $S_t$,
\begin{equation}\label{eq:expr1}
	|S_t| =\sum_{k=1}^m\int_{a_k}^{b_k}
	|1-t\kp(s)|\dd s = 
	\sum_{k=1}^m\int_{a_k}^{b_k}
	(1-t\kp(s))\dd s. 
\end{equation}
Let us define the points 
\[
	\frp_k = \gamma(a_k) +t\nb(a_k),\qquad \frq_k := \gamma(b_k) + t\nb (b_k),\qquad 
	\text{for}\,\,\, k\in\{1,2,\dots,m+1\},
\]	
and the line segments connecting them
\begin{equation}\label{key}
	\cI_k := \{(1-s) \frq_k + s \frp_{k+1}\colon s\in[0,1]\},\qquad k\in\{1,2,\dots,m\}.
\end{equation}
In the case that $b_m = L$,  the line segment $\cI_m$ reduces to a single point.  Note that, even when
\(b_k<a_{k+1}\),  the line segment \(\cI_k\) could reduce to a single point. 
The piecewise smooth continuous mapping $\s_t\colon\dR/(L\dZ)\arr\dR^2$ defined by
\[
	\s_t(s) =\begin{cases}
	\gamma(s) + t\nb(s),& s\in[a_k,b_k]\qquad\quad\text{for}\, k\in\{1,\dots,m\},\\
	\frac{a_{k+1}-s}{a_{k+1}-b_k}\frq_k + \frac{s-b_k}{a_{k+1}-b_k}\frp_{k+1} ,& s\in [b_k,a_{k+1}]\qquad\text{for}\, k\in\{1,\dots,m\},
\end{cases}
\]
parametrizes a closed, not necessarily simple curve in $\dR^2$. The property (i) in the formulation of Theorem~\ref{thm:main} follows from Proposition~\ref{prop:parallelcurve}
and the construction of the mapping $\s_t$.

Note that the curve $\gg([b_k,a_{k+1}])$, $k\in\{1,\dots,m\}$, satisfies Hypothesis~\ref{hyp} with $c_1 = \frq_k$ and $c_2 = \frp_{k+1}$. Thus, it follows from Proposition~\ref{prop:geometric_bound}
and since $\gamma$ is parametrized by arc-length that
for any $k\in\{1,2,\dots,m\}$
\begin{equation}\label{eq:bnd1}	
    |\cI_k| \le |\gamma([b_k,a_{k+1}])| - t\int_{b_k}^{a_{k+1}}\kp(s)\dd s\\
	 = 
	a_{k+1}-b_k - t\int_{b_k}^{a_{k+1}}\kp(s)\dd s.
\end{equation} 
Combining the formula~\eqref{eq:expr1} for the length of $S_t$ with the upper bounds on the lengths of line segments $\{\cI_k\}_{k=1}^m$ we get
\[
\begin{aligned}
	\ell(\s_t) &= |S_t|+\sum_{k=1}^m|\cI_k|
	\le
	\sum_{k=1}^m\int_{a_k}^{b_k}
	(1-t\kp(s))\dd s+ \sum_{k=1}^m\int_{b_k}^{a_{k+1}}(1-t\kp(s))\dd s\\
	 &\le L -t\int_0^L\kp(s) \dd s = L -2\pi t,
\end{aligned}
\]
where we used the total curvature identity~\eqref{eq:tot-curv} in the last step. Hence, we get the property (ii) in the formulation of Theorem~\ref{thm:main}.

Finally,   if \(\Omg\) is centrally symmetric,  then by Lemma~\ref{lem:symmetry} so is \(S_t\).
Consequently,  to every piece
\(C_k\) of \(S_t\)  joining \(\frp_k\) to \(\frq_k\),  there corresponds a curve \(C_{k^*}\) which is the symmetric of \(C_k\) about the origin.  This forces the number \(m\) of the curves \(C_k\) to be even, unless it is equal to one,  and therefore we get that  the   corresponding joining segments  \((\cI_k)_{1\leq k\leq m}\)
 constitute a centrally symmetric set.  This proves that the image of \(\s_t\) is centrally symmetric, thereby establishing (iii) in the formulation of Theorem~\ref{thm:main}. This proves Theorem~\ref{th:curveconstr}.

\section{Proof of Theorem~\ref{th:isomom} -- An isoperimetric inequality for moments of inertia}\label{sec:proofmomin}
Let $t\in\cL$ and the mapping $\s_t$ be as constructed in the proof of Theorem~\ref{thm:main},  which defines a closed curve \(\Sigma_t\). Since the moment of inertia of a curve about a point $p$ is minimal when $p$ is the centroid of the curve, it suffices to prove \eqref{eq:isomom} with $c(t)$ the centroid of $\Sigma_t$.   Let us introduce the notation  $L_t := \ell(\s_t)$ and re-parametrize the curve \(\Sigma_t\)
by the arc-length via the mapping $\wt\s_t\colon\dR/(L_t\dZ)\arr\dR^2$.   Clearly, we have $\wt\s_t\in H^1(\dR/(L_t\dZ))$ thanks to the regularity of $\s_t$. Furthermore,   by centering the coordinates at   the centroid $c(t)$ of $\Sigma_t$, we can assume that $c(t)=0$, and consequently
 	\begin{equation}\label{eq:orth}
		\int_0^{L_t}\wt\s_t(s)\dd s  = 0.
	\end{equation}
	Using the inclusion $S_t\subset\wt \Sigma_t$ and
	and applying the Wirtinger inequality~\cite[\S 7.7]{HLP},
	\begin{equation}\label{eq:Wirtinger}\int_{0}^{L_t}|\wt\s_t(s)|^2\dd s
		\\
		\le 
		\frac{|L_t|^2}{4\pi^2}
		\int_{0}^{L_t}|\wt\s_t'(s)|^2\dd s ,\end{equation}
	 we get
	\begin{equation}\label{eq:proof-isomom}
		\int_{S_t}|x|^2\dd \cH^1(x) \le 
		\int_{0}^{L_t}|\wt\s_t(s)|^2\dd s
		\\
		\le 
		\frac{|L_t|^3}{4\pi^2} \le \frac{(L-2\pi t)^3}{4\pi^2},		
	\end{equation}
	where we employed that $L_t \le L-2\pi t$ in the last step. Therefore, \eqref{eq:isomom} is proved.

Assuming that there is equality in \eqref{eq:isomom},  then we get from \eqref{eq:proof-isomom} that $L_t=L-2\pi t$ and there is equality is \eqref{eq:Wirtinger}. Under the conditions \eqref{eq:orth} and $|\tilde\sigma_t'(s)|=1$, equality happens in \eqref{eq:Wirtinger} if and only if $\tilde\sigma_t(s)=\frac{L_t}{2\pi}\ee^{\pm\ii 2\pi (s-s_0)/L_t}$ for some $s_0\in\R$ and $\Sigma_t$ is a circle. Moreover,   knowing that $L_t=L-2\pi t$ and $S_t=\Sigma_t$,
we get that $S_t$ is the circle of center $0$ and radius $\frac{L}{2\pi}- t$, and consequently,  the domain $\Omega$ with perimeter $L$  contains the disk $\cB$ of radius $L/2\pi$, hence $\Omg=\cB$  is a disk, thanks to the geometric isoperimetric inequality.

Finally, if $\Omg$ is a disk, then $S_t$ is a circle of radius $
\frac{L}{2\pi}- t$ and  equality in \eqref{eq:isomom} occurs.

\section{Proof of Theorem \ref{th:pthmoments} -- An isoperimetric inequality for $p$-th moments}\label{sec:proofpmom}

In this subsection, we study $p$-th moments of inner parallel curves and prove Theorem~\ref{th:pthmoments}. 
We show Theorem~\ref{th:pthmoments} (i) (extension of Theorem~\ref{th:isomom} to $p$-th moments, for $0\leq p\leq 2$) in Corollary \ref{co:iso-p-m} with the help of Jensen's inequality. Proposition \ref{prop:p-moment} yields Theorem~\ref{th:pthmoments} (ii) and (iii).
\begin{corollary}\label{co:iso-p-m}
Let $p\in[0,2]$ and suppose that $\Omega \subset \mathbb{R}^{2}$ is a smooth, bounded and simply connected domain. Then,  for almost every $t\ge0$
\begin{align}\label{eq:iso-p-m}
\int_{S_{t}(\Omg)}|x-c(t)|^{p} \dd\mathcal{H}^{1}(x) \leq \int_{S_{t}(\cB)}|x|^{p} \dd\mathcal{H}^{1}(x),
\end{align}
where $c(t)\in\dR^2$ is the centroid of $S_t(\Omg)$ and
where $\cB$ is the disk centered at the origin and with the same perimeter as $\Omg$.
Here $S_{t}(\cdot)$ is defined in \eqref{eq:stdef},  and $\mathcal{H}^{1}$ is the one-dimensional Hausdorff measure. For $p\not=0$ and $t\in[0,\frac{|\p\Omg|}{2\pi})$, the equality in \eqref{eq:iso-p-m}
is attained if and only if $\Omega$ is a disk 
\end{corollary}
\begin{proof}
For $p=0$, \eqref{eq:iso-p-m} reduces to the well known bound $|S_t|\leq L-2\pi t$, see \eqref{eq:sthartest}.

Let us take $p\in(0,2]$. 
Since $2/p\geq 1$, we write by Jensen's inequality,
\[\left(\int_{S_{t}(\Omg)}|x-c(t)|^{p} \dd\mathcal{H}^{1}(x)\right)^{2/p}\leq |S_t(\Omg)|^{\frac2p-1}\int_{S_t(\Omg)}|x-c(t)|^{2} \dd\mathcal{H}^{1}(x).\]
To finish the proof, we use that $|S_t|\leq L-2\pi t$, apply Theorem~\ref{th:isomom},  and note that 
\[\int_{S_t(\cB)}|x|^p \dd\mathcal{H}^{1}(x)=2\pi\left(\frac{L}{2\pi}-t \right)^{1+p}. \]
Since the inequality in Theorem~\ref{th:isomom} is strict for all $t\in [0,\frac{|\p\Omg|}{2\pi})$ when $\Omg$ is not a disk, this also holds for \eqref{eq:iso-p-m}. 
\end{proof}
Secondly, we formulate the corresponding variational problem to Theorem~\ref{th:pthmoments}.
\begin{definition}
Given $p>0$, we define
\begin{align*}
 C_{p}:=\sup _{\Gamma} \frac{\displaystyle\int_{\Gamma}|x|^{p} \dd \mathcal{H}^{1}(x)}{\displaystyle\int_{\partial \cB}|x|^{p} \dd \mathcal{H}^{1}(x)}=\sup _{\Gamma} \frac{(2 \pi)^{p}}{|\Gamma|^{p+1}} \int_{\Gamma}|x|^{p} \dd \mathcal{H}^{1}(x) \text {, }
\end{align*}
where the supremum is taken over all centrally symmetric, closed Lipschitz curves $\Gamma$ and $\cB$ is a disk centred at the origin with $|\partial \cB|=|\Gamma|$. 
\end{definition}

\begin{remark}
By scaling, we find that $C_{p} \in(0, \infty)$. By testing with $\Gamma= \partial\cB$, we find $C_{p}\ge1$ for all $p$. We have already shown that $C_{p}=1$ for all $p \in(0,2]$, see the end of the proof of Theorem~\ref{th:isomom} (or alternatively the result by Hurwitz \cite{Hu}) combined with Jensen's inequality as in the proof of Corollary \ref{co:iso-p-m}. In fact, using Jensen's inequality as in the proof of Corollary~\ref{co:iso-p-m}, one can show that $C_{p}$ is non-decreasing in $p$. 
\end{remark}

The next proposition shows that $C_{p}$ is not constant.
\begin{proposition}
    We have $C_{p}>1$ for all $p$ large enough.
\end{proposition}
\begin{proof}
Consider the curve $\Gamma$ parametrised by $\gamma\colon[0,1] \longrightarrow \mathbb{R}^{2}$, $\gamma(s)=\left(\gamma_{1}(s), \gamma_{2}(s)\right)^{\top}$ with $\gamma_{2}(s)=0$ for all $s\in[0,1]$ and
\begin{align*}
\gamma_{1}(s)=\left\{\begin{array}{lll}
-\frac{1}{4}+s & \text { for } & s \in\bigl[0, \tfrac{1}{2}\bigr]\medskip \\
\frac{3}{4}-s & \text { for } & s \in\bigl[\tfrac{1}{2}, 1\bigr]
\end{array}\right.
\end{align*}
$\Gamma$ is a closed curve that is parametrised by arc-length, centrally symmetric, its shape is the doubly covered interval $\left[-\frac{1}{4}, \frac{1}{4}\right] \times\{0\}$, and its length is $|\Gamma|=1$.
We have
\begin{align*}
 \lim _{p \rightarrow \infty}\left(\int_{\Gamma}|x|^{p} \dd \mathcal{H}^{1}(x)\right)^{1 / p}=\lim _{p \rightarrow \infty}\left(\int_{0}^{1}\left|\gamma_{1}(s)\right|^{p} \dd s\right)^{1 / p} 
 =\sup _{s \in[0,1]}\left|\gamma_{1}(s)\right|=\frac{1}{4},
\end{align*}
and therefore, by $|\Gamma|=1$,
\begin{align*}
\lim _{p \rightarrow \infty}\left(\frac{(2 \pi)^{p}}{|\Gamma|^{p+1}} \int_{\Gamma}|x|^{p} \dd \mathcal{H}^{1}(x) \right)^{1 / p}=\frac{2 \pi}{4}>1 \text {. }
\end{align*}
This proves $C_{p}>1$ for all $p$ large enough.
\end{proof}

This leads to the question of determining the critical value
\begin{align*}
p_{*}:=\sup \left\{p \mid p>0 \text { and } C_{p}=1\right\}
\end{align*}
The following proposition shows that $p_*\le3$, see (ii), which we conjecture to be optimal (see Conjecture \ref{con:p3}) since the disk is a local optimiser, see (i) below.
\begin{proposition}\label{prop:p-moment}
    \begin{enumerate}
        \item[(i)] Let $p<3$. Then the disk is a local optimiser among centrally symmetric curves in the following sense: If $r: \mathbb{R} / 2 \pi \mathbb{Z} \rightarrow \mathbb{R}$ is continuous with $r(\theta)=r(\theta+\pi)$ for all $\theta$, and the curve $\Gamma_{\varepsilon}$ is parametrised by $\gamma_{\varepsilon}: \mathbb{R} / 2 \pi \mathbb{Z}  \rightarrow \mathbb{R}^{2}$ with
\begin{align}\label{eq:Reps}
\gamma_{\varepsilon}(\theta)=\left(\begin{array}{l}
R_{\varepsilon}(\theta) \cos (\theta) \\
R_{\varepsilon}(\theta) \sin (\theta)
\end{array}\right) \text {, where } R_{\varepsilon}(\theta)=1+\varepsilon r(\theta) \text {, }
\end{align}
then
\begin{align*}
\frac{(2 \pi)^{p}}{\left|\Gamma_{\varepsilon}\right|^{p+1}} \int_{\Gamma_{\varepsilon}}|x|^{p} \dd \mathcal{H}^{1}(x) \leq 1
\end{align*}
for all $\varepsilon>0$ small enough; furthermore, the inequality is strict if $r$ is non-constant.
\item[(ii)] 
Let $p>3$. Then the disk is not optimal, not even locally: There exists a  sequence of nearly circular, centrally symmetric closed Lipschitz curves $\left(\Gamma_n\right)_{n\in\N}$ converging uniformly to the boundary of the disk for which
\begin{align*}
\frac{(2 \pi)^{p}}{\left|\Gamma_{n}\right|^{p+1}} \int_{\Gamma_{n}}|x|^{p} \dd \mathcal{H}^{1}(x)>1.
\end{align*}
    \end{enumerate}
\end{proposition}
\begin{proof}
Consider the curve $\G_\varepsilon$ defined by the  parametrization of $\G_\varepsilon$ in \eqref{eq:Reps}. Since we assume that $r(\theta)=r(\theta+\pi)$, the curve $\G_\varepsilon$ is centrally symmetric.  

It is straightforward to check that with $\gamma_\varepsilon$  defined as in \eqref{eq:Reps},
\[
\begin{aligned}
|\gg_\varepsilon(\theta)|^p&=1+\varepsilon pr(\theta)+\frac{\varepsilon^2}{2}p(p-1)|r(\theta)|^2+\mathcal O(\varepsilon^3),\\
|\gamma_\varepsilon'(\theta)|&=1+\varepsilon r(\theta)+\frac{\varepsilon^2}{2}|r'(\theta)|^2+\mathcal O(\varepsilon^3),\\
|\G_\varepsilon|&=2\pi+\varepsilon\int_0^{2\pi}r(\theta)\dd\theta+\frac{\varepsilon^2}2\int_0^{2\pi}|r'(\theta)|^2\dd\theta+\mathcal O(\varepsilon^3).
\end{aligned}
\]
With the above formulas in hand, we get
\begin{multline*}
\frac{|\G_\varepsilon|^{p+1}}{(2\pi)^p}=2\pi+\varepsilon (p+1)\int_0^{2\pi}r(\theta)\dd\theta\\
+\varepsilon^2(p+1)\Bigl(\frac12\int_0^{2\pi}|r'(\theta)|^2\dd\theta+
\frac{p}{4\pi}\Bigl(\int_0^{2\pi}r(\theta)\dd\theta\Bigr)^2\Bigr)+\mathcal O(\varepsilon^3),    
\end{multline*}
and
\begin{equation}\label{eq:gecompe}
   \int_{\G_\varepsilon}|x|^p\dd\cH^1(x)-\frac{|\G_\varepsilon|^{p+1}}{(2\pi)^{p}}=\frac{\varepsilon^2p}{2}\Big(\cF(r)+\mathcal O(\varepsilon)\Bigr), 
\end{equation}
where
\[
\begin{aligned}
\mathcal F(r)&=(p+1)\int_0^{2\pi}|r(\theta)|^2\dd\theta-\frac{p+1}{2\pi}\Bigl(
\int_0^{2\pi}r(\theta)\dd\theta\Bigr)^2-\int_0^{2\pi}|r'(\theta)|^2\dd\theta\\
&=\int_0^{2\pi} \Bigl[(p+1)\Bigl(r(\theta)-\frac{1}{2\pi}\int_0^{2\pi}r(\theta)\dd\theta\Bigr)^2-|r'(\theta)|^2\Bigr]\dd\theta.
\end{aligned}\]
We expand $r$ as a Fourier series and notice that the coefficients of the odd indices will vanish, thanks to the symmetry condition on $r$. More precisely, we have
\[\begin{aligned}
r(\theta)-\frac{1}{2\pi}\int_0^{2\pi}r(\theta)\dd\theta
&=\sum_{n\geq 2}a_n\cos(n\theta)+\sum_{n\geq 2}b_n\sin(n\theta),\\
r'(\theta)&=-\sum_{n\geq 2}na_n\sin(n\theta)+\sum_{n\geq 2}nb_n\cos(n\theta).
\end{aligned}\]
By Parseval's identity, we write
\[\begin{aligned}
\int_0^{2\pi}\Bigl(r(\theta)-\frac{1}{2\pi}\int_0^{2\pi}r(\theta)\dd\theta\Bigr)^2\dd\theta
&=\pi\sum_{n\geq 2}\bigl(|a_n|^2+|b_n|^2\bigr),\\
\int_0^{2\pi}|r'(\theta)|^2\dd\theta &=\pi\sum_{n\geq 2}n^2\bigl(|a_n|^2+|b_n|^2\bigr).
\end{aligned}\]
Hence,
\begin{equation}\label{eq:four4}
   \int_0^{2\pi}|r'(\theta)|^2\dd\theta\geq 4\int_0^{2\pi}\Bigl(r(\theta)-\frac{1}{2\pi}\int_0^{2\pi}r(\theta)\dd\theta\Bigr)^2\dd\theta, 
\end{equation}
and consequently, for $p< 3$
\[\cF(r)\leq (p-3)\int_0^{2\pi}\Bigl(r(\theta)-\frac{1}{2\pi}\int_0^{2\pi}r(\theta)\dd\theta\Bigr)^2\dd\theta\le0.\]
For $p<3$, $\cF(r)$ vanishes if and only if the function $r$ is constant. To conclude the proof of (i), we take $\varepsilon\to0$ and note that  the disk is a strict local maximiser if and only if we have for all non-constant $r$ and $\varepsilon$ small enough
\begin{equation*}
     \int_{\G_\varepsilon}|x|^p\dd\cH^1(x)-\frac{|\G_\varepsilon|^{p+1}}{(2\pi)^{p}}<0.
\end{equation*}
Taking $\varepsilon\to0$ and using \eqref{eq:gecompe}, we obtain the desired result.

\bigskip

For (ii), note that choosing $r(\theta):=\sin(2\theta)$, we have equality in \eqref{eq:four4}, which leads to
\begin{equation}
    \cF(r)= (p-3)\int_0^{2\pi}\Bigl(r(\theta)-\frac{1}{2\pi}\int_0^{2\pi}r(\theta)\dd\theta\Bigr)^2\dd\theta.
\end{equation}
And for $p>3$, we get $\cF(r)>0$, which yields the claim by \eqref{eq:gecompe} for $\varepsilon$ small enough.
\end{proof}

\section{Proof of Theorem \ref{th:impmagintro} -- Applications to the magnetic Robin Laplacian}\label{sec:magnetic}
In this subsection, we show that Theorem \ref{th:isomom} can be used to relax the assumptions on the domain in the isoperimetric inequality for the lowest eigenvalue of the magnetic Robin Laplacian on a bounded domain with a negative boundary parameter, recently obtained in~\cite{KL} by the second and the third authors of the present paper. 

The operator we study involves 
the vector potential (magnetic potential)
\begin{equation}\label{eq:mp}
\Ab(x):=\frac12(-x_2,x_1)^\top,\quad\big( x=(x_1,x_2)\big)\,.
\end{equation}
and two parameters,  $b\geq 0$ standing for the intensity of the magnetic field and $\beta\le0$, the Robin parameter, appearing in the boundary condition. Let $\Omg\subset\dR^2$ be a bounded simply-connected smooth domain. Our magnetic Robin Laplacian, $\sfH_{\Omg}^{\beta,b}$, is  the self-adjoint operator defined by the closed, symmetric, densely defined and lower-semibounded quadratic form
\begin{equation}\label{eq:qf}
\frq_{\Omega}^{\beta,b}[u] :=\|(\nabla-\ii b\Ab)u\|^2_{L^2(\Omg;\dC^2)}+\beta\|u\|_{L^2(\partial\Omega)}^2,\quad \dom\frq_{\Omg}^{\beta,b} := H^1(\Omg),
\end{equation}
and it is characterized by
\[
\begin{split}
\dom\sfH_{\Omg}^{\beta,b} &\! = \!
\big\{
u\!\in\! H^1(\Omg)\colon
\exists\, w\in L^2(\Omg):
\frq_{\Omg}^{\beta,b}[u,v] = ( w, 
v)_{L^2(\Omg)}, \forall\,v\in\dom\frq_{\Omg}^{\beta,b}
\big\},\\
\sfH_{\Omg}^{\beta,b} u&:=-(\nabla-\ii b\Ab)^2 u= w\,.
\end{split}
\]
Denoting by $\nu$ the unit inward normal vector on $\p\Omg$, we observe that functions in $\dom\sfH_{\Omg}^{\beta,b}$ satisfy  the (magnetic) Robin boundary condition
\[\nu\cdot(\nabla-\ii b\Ab)u=\beta u\,\,\,\,{\rm on}\,\,\,\p\Omg\,.\]
The isoperimetric inequality obtained in \cite{KL} concerns the lowest eigenvalue of 
$\sfH_{\Omg}^{\beta,b}$, which we express in the variational  form  as follows 
\begin{equation}\label{eq:ev1}
\lambda_1^{\beta,b}(\Omega) {:=} \inf_{u\in H^1(\Omega)\setminus\{0\}} \frac{ {\frq_{\Omega}^{\beta,b}[u]}}{\|u\|^2_{L^2(\Omg)}}\,.
\end{equation}
Denoting by $\cB$ the  disk in $\R^2$ centered at the origin, with radius $R$  and having  the  same perimeter $ 2\pi R=|\partial\Omega|$ as the domain $\Omg$, it is known that the following inequality holds (see \cite[Theorem 4.8, Corollary~4.9]{KL})
\begin{equation}\label{eq:iso-p-ine}
\lambda_1^{\beta,b}(\Omega)\leq \lambda_1^{\beta,b}(\cB)\,,
\end{equation}
provided that 
\begin{itemize}
	\item[(i)] $\beta<0$ and $0<b<\min\big(R^{-2},4\sqrt{-\beta}\,R^{-3/2}\big)$ (i.e. the magnetic field's intensity $b$ is of moderate strength); and
	\item[(ii)] The \emph{inner parallel curves} of  $\Omg$ obey  the  condition \begin{equation}\label{eq:condition_St}	
	\int_{S_t}|x-x_0|^2\dd\cH^1(x) \le \frac{(L-2\pi t)^3}{4\pi^2}
	\end{equation}
	for some fixed point $x_0\in\dR^2$ and almost all $t\in(0,r_{\rm i}(\Omg))$.
	This condition holds for instance, when $\Omg\subset\cB$ or when $\Omg$ is  \emph{convex} and \emph{centrally symmetric} (see \cite[Proposition 4.4]{KL})
\end{itemize}
\begin{proof}[Proof of Theorem \ref{th:impmagintro}]
In view of Theorem~\ref{th:isomom} the condition in~\eqref{eq:condition_St} holds
with $x_0=0$ for all bounded centrally symmetric simply-connected smooth domains or, more generally, with $x_0$ being the centroid of all $S_t(\Omg)$ for all simply-connected smooth domains $\Omg$ such that the centroid of the inner parallel curve $S_t(\Omg)$ is independent of $t$. Thus, we relaxed the convexity assumption on the domain $\Omg$. We obtain Theorem \ref{th:impmagintro} with the choice $b_0(|\partial\Omega|, \beta)= \min\{R^{-2},4\sqrt{-\beta} R^{-3/2}\}$, where $R:=\frac{|\partial\Omega|}{2\pi}$.
\end{proof}
\section{Some direct consequences of Theorems \ref{th:isomom} and \ref{th:curveconstr}}\label{sec:appl}
\subsection{A refined bound on the length of the disconnected inner parallel curve}
\label{ssec:refinement}
In this subsection we use Theorem~\ref{thm:main} to get a refined upper bound on the length of the inner parallel curve $S_t$ in the situation when $S_t$ consists of several connected components

Let $\Omg\subset\dR^2$ be a bounded simply-connected smooth domain
with perimeter $L > 0$. Let the inner parallel curve $S_t\subset\Omg$ be as in~\eqref{eq:stdef}. For any $t\in\cL$ we have by Proposition~\ref{prop:parallelcurve}
for some $N\in\dN$
\[
	S_t = \bigcup_{n=1}^N \G_n,\qquad |S_t| \le L -2\pi t,
\]
where $\{\G_n\}_{n=1}^N$ are piecewise-smooth closed simple curves
 that pairwise disjoint.  In the case that $S_t$ is connected, one has $N = 1$. However, in general $N$ can be an arbitrarily large integer number. 
In the case that $N = 2$ we immediately get as a consequence of Theorem~\ref{thm:main}
\[
	|S_t| + 2\dist(\G_1,\G_2) \le L-2\pi t.
\]
This observation can be generalized to the case of arbitrary $N\in\dN$ to  improve Hartman's bound \eqref{eq:sthartest}, see \cite{H64}, on the length of the inner parallel curve \(S_t\).
\begin{corollary}\label{co:hartimp}
    For all $t\in\cL$, it holds that
\[
	|S_t| + \sum_{n=1}^N\dist(\G_n,S_t\sm\G_n) \le L-2\pi t.
\]
\end{corollary}
\begin{proof}
	The length of the closed piecewise-smooth curve parametrized by the mapping $\s_t$
	constructed in the proof of Theorem~\ref{thm:main} is given by
	\[
		\ell(\s_t) = |S_t| + \sum_{k=1}^m|\cI_k|.
	\]
 Every $\cI_k$ connects some $\Gamma_{n(k)}$ with some $\Gamma_{n(k+1)}$, and for each $n\in\{1,\dots,N\}$ there is at least one  \(k\in \{1,\dots,m\}\) such that $n=n(k)$. 
	  Hence, we get that
   	\[	
		\sum_{k=1}^m|\cI_k| \ge\sum_{k=1}^m \dist(\Gamma_{n(k)}, \Gamma_{n(k+1)})\ge\sum_{n=1}^N \dist(\G_n, S_t\sm\G_n).
	\]
	Thus, we conclude that
	\[
		|S_t|
		+ \sum_{n=1}^N \dist(\G_n, S_t\sm\G_n)\le \ell(\s_t) \le L-2\pi t.\qedhere
	\]
\end{proof}

\subsection{Moments of inertia of domains}\label{ssec:moments}

In this subsection, we apply Theorem~\ref{th:isomom} to recover an isoperimetric upper bound on the moment of inertia for the domain $\Omg$ itself leading to an alternative proof of a result due to Hadwiger~\cite{H56}.

\begin{corollary}
    \label{co:integriso}
Let $\Omega \subset \mathbb{R}^{2}$ be a smooth, bounded, simply connected domain. Assume that $\Omg$ is centrally symmetric or, more generally, that the origin is the centroid of $S_t$  for almost every $t\in (0,r_{\rm i}(\Omg))$. Then
\begin{align}
\int_{\Omega}|x|^{2} \dd x \leq \int_{\cB}|x|^{2} \dd x,
\end{align}
where $\cB \subset \mathbb{R}^{2}$ is a disk centered at the origin with the same perimeter as $\Omega$.
\end{corollary}
\begin{proof}
Recall that $\Omg\subset\dR^2$ is a bounded simply-connected smooth domain with the perimeter $L > 0$ and the origin
being the centroid of $S_t$ for almost every $t\in (0,r_{\rm i}(\Omg))$, and that $\cB\subset\dR^2$ is the disk of radius $R = \frac{L}{2\pi}$, having thus the same perimeter as $\Omg$. By the geometric  isoperimetric inequality we have $|\Omg|\le |\cB|$ and therefore it holds that $R\ge r_{\rm i}(\Omg)$.

\medskip

Recall  the co-area formula    in two dimensions  (see~\cite[Thm. 4.20]{B19} and~\cite{MSZ02}).  If $\cA\subset\dR^2$ is an open set,  $f\colon\cA\arr\dR$ is a Lipschitz continuous real-valued function,   and  $g\colon\cA\arr\dR$ is an integrable function,  then we have
\begin{equation}\label{eq:coarea}
\int_\cA g(x)|\nabla f(x)|\dd x =
\int_\dR\int_{f^{-1}(t)} g(x)\dd \cH^1(x)\,\dd t.
\end{equation}
Applying the co-area formula~\eqref{eq:coarea} with $\cA = \Omg$, $g(x) = |x|^2$ and $f(x) = \rho(x)$ (the distance function to the boundary of $\Omg$ defined in \eqref{eq:distance}) we get using the inequality in Theorem~\ref{th:isomom},
\begin{equation}\label{eq:moment_isoperimetric}
\begin{aligned}
 \int_\Omg |x|^2\dd x &= \int_0^{r_{\rm i}(\Omg)}\int_{S_t}|x|^2\dd\cH^1(x)
	\le\int_0^{r_{\rm i}(\Omg)}\frac{(L-2\pi t)^3}{4\pi^2}\dd t\\
	&\le \int_0^{R}\frac{(L-2\pi t)^3}{4\pi^2}\dd t = 
	2\pi\int_0^R(R-t)^3\dd t = \frac{\pi R^4}{2} = \int_{\cB}|x|^2\dd x
 .\qedhere
\end{aligned}
\end{equation}
\end{proof}
\begin{remark}
The isoperimetric inequality \eqref{co:integriso} 
% \(I_0(\Omg)\leq I_0(\cB)\)
in the case of centrally symmetric domains can  be alternatively derived from the inequality by Hadwiger~\cite{H56}, where only convex domains were considered.   Let $\cK\subset\dR^2$ be a bounded convex domain with a Lipschitz boundary.  We can translate the domain $\cK$ so that the origin becomes the centroid of $\cK$ in the sense that
\[
	\int_{\cK} x \dd x = 0.
\]
Let $\cB'\subset\dR^2$ be the disk
centred at the origin of the same perimeter as $\cK$.
It is proved in~\cite{H56} that 
\begin{equation}\label{eq:Hadwiger}
\int_{\cK}|x|^2\dd x \le \int_{\cB'}|x|^2\dd x.
\end{equation}
Let us define the domain $\cK$ as the convex hull of the bounded simply-connected centrally symmetric smooth  $\Omg\subset\dR^2$. Then, the perimeter of $\cK$ does not exceed the perimeter of $\Omg$. This is a well-known fact, whose proof can be found \emph{e.g.} in~\cite{T02}. Moreover, the convex domain $\cK$ is centrally symmetric as well. Therefore, the origin is the centroid of $\cK$. Hence, we get from~\eqref{eq:Hadwiger}
\[
 \int_{\Omg}|x|^2\dd x\le \int_{\cK}|x|^2\dd x\le
	\int_{\cB'}|x|^2\dd x\le \int_{\cB}|x|^2\dd x ,
\]  
where we used that the perimeter of \(\Omg\) is larger than or equal to the perimeter of \(\cK\), so  the radius of $\cB'$ does not exceed the radius of $\cB$.
\end{remark}
\subsection*{Acknowledgements} 
CD would like to express her deepest gratitude to Phan Th\`anh Nam for his continued support and very helpful advice. She would also like to thank Laure Saint-Raymond for her support, inspiring discussions and hospitality at Institut des Hautes \'Etudes Scientifiques, where large parts of this work were carried out. 
\medskip

Special thanks go to Mirek Ol\v{s}\'{a}k for his ideas and discussions on an elementary proof of Proposition~\ref{prop:geometric_bound}, see Appendix \ref{sec:geom_ineq}. CD would like to thank Prasuna Bandi, Sami Fersi, Matthias Paulsen, and Shuddhodan Kadattur Vasudevan, and Leonard Wetzel for helpful discussions on Theorem \ref{th:curveconstr}. Furthermore, she would like to thank Bo Berndtsson, Jan Derezi\'nski, Dmitry Faifman, Rupert Frank, Jonas Peteranderl, Heinz Siedentop, and Jakob Stern for helpful discussions on Theorem \ref{th:pthmoments} and further directions. She would also like to thank Larry Read for helpful discussions and his help in creating the images for the introduction.
AK and CD wish to thank P. Mironescu for discussions around \cite{BM}.
\medskip

CD acknowledges the support from the Deutsche Forschungsgemeinschaft (DFG project Nr.~426365943), from the Jean-Paul Gimon Fund, and from the Erasmus+ programme.
AK acknowledges the  support from The Chinese University of Hong Kong, Shenzhen (grant UDF01003322), and from the Knut and Alice Wallenberg Foundation (grant KAW 2021.0259). VL acknowledges the support by the grant No.~21-07129S of the Czech Science Foundation (GA\v{C}R).
\bigskip

\appendix

\section{An elementary proof of an auxiliary geometric inequality }\label{app}
The ultimate goal of this appendix is to prove Proposition~\ref{prop:geometric_bound}.
In Subsection~\ref{sec:winding}, we obtain additional properties of the winding numbers of non-closed curves in the plane, which are extensively used in Subsection~\ref{ssec:proof}, where we present an elementary proof of Proposition~\ref{prop:geometric_bound}, due to Mirek Ol\v{s}\'{a}k.  
\subsection{Winding numbers}\label{sec:winding}
In this subsection we recall the concept of a winding number of a non-closed curve in the plane and provide some of its properties.  In particular, we establish a connection between the winding numbers and the total curvature of a non-closed curve. Further details can be found e.g. in~\cite[Chapter 2]{ASS}. 

First, we define the winding number of a non-closed curve with respect to a point in the plane not lying on this curve. 
\begin{definition}\label{dfn:winding}
	Let $\Gamma\subset\R^2$
	be a simple smooth non-closed curve with endpoints $x_1,x_2\in\R^2$, $x_1\ne x_2$, parametrized by the mapping $\gamma\colon[s_1,s_2]\rightarrow\R^2$, ($s_1 < s_2$) with $\gamma\in C^\infty([s_1,s_2];\R^2)$
	such that $|\gamma'(s)|= 1$ for all $s\in[s_1,s_2]$ and that $\gg(s_j) = x_j$ for $j=1,2$.  Let the point $x_0\in\R^2\setminus\Gamma$ be fixed.
	By~\cite[Corollary 2.2]{ASS} there exists a unique continuous angle function $\Th\colon[s_1,s_2]\rightarrow\R$ with $\Th(s_1)\in [0,2\pi)$ such that 
	\begin{equation}\label{eq:angle_function}
	\gamma(s) = x_0+ |\gamma(s)-x_0|\big(\cos\Th(s),\sin\Th(s)\big)^\top, \qquad\text{for all}\,\, s\in [s_1,s_2].
	\end{equation}
	Then we define the winding number of $\Gamma$ with respect to $x_0$ by
	\[
	w_\Gamma(x_0) := \Th(s_2) - \Th(s_1).
	\]
\end{definition}
\begin{remark}\label{rem:sign-wn}
	The sign of the winding number in Definition~\ref{dfn:winding} depends
	on the choice of the starting and the terminal points of the curve $\G$. In the following, we will always specify the order of the endpoints.
\end{remark}
\begin{remark}\label{rem:phase-wn}
The condition \(\Th(s_1)\in[0,2\pi)\) ensures the uniqueness of \(\Th\),  but is unnecessary to define the winding number.  In fact,  if \(\widetilde\Th:[s_1,s_2]\to\mathbb R\) is a continuous function such that  \(\gamma(s) -x_0= |\gamma(s)-x_0|\big(\cos\widetilde\Th(s),\sin\widetilde\Th(s)\big)^\top\),  then  \(\widetilde\Th\) is the same as \(\Th\)  up to a shift by an integer multiple of \(2\pi\),  and consequently \(w_\Gamma(x_0) := \widetilde\Th(s_2) -\widetilde \Th(s_1)\).
\end{remark}

Let $\Gamma_1,\Gamma_2\subset\R^2$ be  non-closed curves parametrized by arc-length via the  mappings $\gamma_1\in C^\infty([s_0,s_1];\R^2)$, $\gamma_2\in C^\infty([s_1,s_2];\R^2)$ such that $\gamma_1(s_1)=\gamma_2(s_1)$ (i.e. the terminal point of $\Gamma_1$ is the starting point of $\Gamma_2$).
Then, for any $x_0\in\R^2\setminus(\Gamma_1\cup\Gamma_2)$ one has by~\cite[Theorem~7.8]{ST}
\begin{equation}\label{eq:additivity}	
w_{\Gamma_1\cup\Gamma_2}(x_0) = w_{\Gamma_1}(x_0) + w_{\Gamma_2}(x_0).
\end{equation}

Definition~\ref{dfn:winding} can be extended to the case when the winding number is computed with respect to one of the endpoints $x_1,x_2$ of the curve $\Gamma$, because the unique continuous angle function parametrizing the curve $\G$ as in~\eqref{eq:angle_function} still exists in this situation.       Here,   the  fact that $\G$ is smooth up to the endpoints is helpful for the existence of such an angle function.   More precisely,   we can extend  $\gamma$ to   $\widetilde\gamma\in C^\infty([s_1-\alpha,s_2+\alpha];\R^2)$,  for some $\alpha>0$,  such that   \(|\widetilde\gamma\,'(s)|=1\) on $[s_1-\alpha,s_2+\alpha]$ and \(\widetilde\Gamma:=\widetilde\gamma([s_1-\alpha,s_2+\alpha])\) is a simple curve.  In particular,  we know that \(\widetilde\gamma(s)\not\in\Gamma\) for \(s<s_1\) and \(s>s_2\).   We then define the winding numbers with respect to \(\gamma(s_1)\) and \(\gamma(s_2)\) as follows
\begin{equation}\label{eq:ext-wn}    
w_{\G}\bigl(\gamma(s_1)\bigr):=\lim_{s\nearrow s_1}w_\Gamma\bigl(\widetilde\gamma(s)\bigr)\quad\mbox{and}\quad w_{\G}\bigl(\gamma(s_2)\bigr):=\lim_{s\searrow s_2}w_\Gamma\bigl(\widetilde\gamma(s)\bigr).
\end{equation}
Let us remark that   the additivity of the winding number under gluing the curves stated in~\eqref{eq:additivity} remains valid upon extension of the definition to the case when the winding number is computed relative to an endpoint of $\G_1\cup\G_2$.   

That the limits in \eqref{eq:ext-wn} do exist and are independent of the choice of the extension \(\widetilde\gamma\) is the subject of the following lemma.
\begin{lemma}\label{lem:ext-wn}
Suppose that $\wG$ is a simple curve parametrized by \( \wg\in C^{\infty}([s_1-\alpha,s_2+\alpha];\mathbb R^2)\) such that $|\wg\,'(s)|=1$ on \([s_1-\alpha,s_2+\alpha]\).  For \(s,t\in(s_1-\alpha,s_2+\alpha)\),  let \(\wG_{s,t}=\wg([s,t])\).  Then,  the following limits exist
\begin{equation}\label{eq:ext-wn*}
\lim_{s\nearrow s_1}w_{\wG_{s_1,s_2}}\bigl(\widetilde\gamma(s)\bigr)\quad\mbox{and}\quad\lim_{s\searrow s_2}w_{\,\wG_{s_1,s_2}}\bigl(\widetilde\gamma(s)\bigr),
\end{equation}
and only depend on \(\wg|_{[s_1,s_2]}\).  Furthermore,  for any \(s_0\in[s_1,s_2]\),  we have
\begin{equation}\label{eq:cont-wn}
\lim_{\substack{\\s\to s_0\\s\in(s_1,s_2+\alpha)}}w_{\,\wG_{s_1,s}}\bigl(\wg(s)\bigr)=w_{\,\wG_{s_1,s_0}}\bigl(\wg(s_0)\bigr),
\end{equation}
with the convention that \(w_{\,\wG_{s_1,s}}\bigl(\wg(s)\bigr)=\lim_{\tilde s\searrow s}w_{\,\wG_{s_1,s}}\bigl(\wg(\tilde s)\bigr)\),  and that the right-hand side in \eqref{eq:cont-wn} is zero for \(s_0=s_1\).
\end{lemma}
We postpone the proof of Lemma~\ref{lem:ext-wn} until the end of this section. In the next proposition, we provide a connection between the winding numbers relative to the endpoints of a non-closed curve and the total curvature of this curve. 
\begin{proposition}\label{prop:curvature_winding}
	Let the smooth mapping $\gamma\colon[s_1,s_2]\rightarrow\R^2$ with $|\gamma'(s)|=1$ parametrize a simple non-closed curve $\Gamma\subset\R^2$. Let $\kappa\colon[s_1,s_2]\rightarrow\R$
	be the signed curvature of $\Gamma$ (defined as in~\eqref{eq:curvature}).   Then the following holds
	\begin{equation}\label{eq:angle}
	\int_{s_1}^{s_2}\kappa(s)\dd s= w_{\Gamma}(\gamma(s_1))+w_{\Gamma}(\gamma(s_2)).
	\end{equation}
\end{proposition}
\begin{proof}~\\
{\bf Step 1.}

We will first prove
\begin{equation}\label{eq-key-s1}
	w_{\Gamma}(\gamma(s_1))+w_{\Gamma}(\gamma(s_2))-\int_{s_1}^{s_2}\kp(s)\dd s\in2\pi\Z.
	\end{equation}
	We will identify $\dR^2$ with $\dC$ via the mapping $\dR^2\ni(x_1,x_2)^\top\mapsto x_1+\ii x_2$.
	Writing \(\tau(s)=\gamma'(s)=\tau_1(s)+\ii\tau_2(s)\),  we get from the Frenet formula \(\tau'(s)=\ii \kp(s)\tau(s)\) and consequently
	\[
	\tau(s)=\tau(s_1)\exp\left(\ii\int_{s_1}^s \kp(\varsigma)\dd\varsigma\right).
	\]
	We introduce 
	\[\alpha(s):=\int_{s_1}^s\kp(\varsigma)\dd\varsigma+\alpha(s_1)\]
 where \(\alpha(s_1)\in[0,2\pi)\) is defined by \(\tau(s_1)= \ee^{\ii \alpha(s_1)}\).
	
	Writing \(\gamma(s)-\gamma(s_1)=\rho_1(s)\ee^{\ii\Theta_1(s)}\) where \(\rho_1(s)=|\gamma(s)-\gamma(s_1)|\) and \(\Theta_1(\cdot)\) is a continuous function with  \(\Theta_1(s_1)\in[0,2\pi)\),  it follows that
	\[\gamma(s)-\gamma(s_1)= (s-s_1)\tau(s_1)+\mathcal O(|s-s_1|^2)=(s-s_1)\ee^{\ii\alpha(s_1)}+\mathcal O(|s-s_1|^2)\quad\mbox{as }s\searrow s_1\]
	and \(\Theta_1(s_1)=\alpha(s_1)\).  In a similar fashion,  writing \(\gamma(s)-\gamma(s_2)=\rho_2(s)\ee^{\ii\Theta_2(s)}\) where \(\rho_2(s)=|\gamma(s)-\gamma(s_2)|\) and \(\Theta_2(\cdot)\) is a continuous function with  \(\Theta_2(s_2)\in[0,2\pi)\),  we get that
	\[
	\Theta_2(s_2)=\alpha(s_2)-\pi~{\rm mod}\,\,2\pi.
	\]
	The identities \(\rho_1(s_2)\ee^{\ii \Theta_1(s_2)}=\gamma(s_2)-\gamma(s_1)=-\rho_2(s_1)\ee^{\ii\Theta_2(s_1)}\) yield \(\Theta_2(s_1)-\Theta_1(s_2)+\pi\in2\pi \Z\). 
	Consequently,
	\[w_{\Gamma}(\gamma(s_1))+w_{\Gamma}(\gamma(s_2))=\Theta_1(s_2)-\Theta_1(s_1)+\Theta_2(s_2)-\Theta_2(s_1)=\alpha(s_2)-\alpha(s_1)~{\rm mod}\,\,2\pi,\]
	which yields \eqref{eq-key-s1}.

{\bf Step 2.}

For \(s\in (s_1,s_2]\),  put
\[ \Psi(s)= w_{\Gamma_{s_1,s}}(\gamma(s_1))+w_{\Gamma_{s_1,s}}(\gamma(s))-\int_{s_1}^{s}\kp(\varsigma)\dd \varsigma,\]	
where \(\Gamma_{s_1,s}\) is the arc of \(\Gamma\) parametrized by \(\gamma|_{[s_1,s]}\). 
Notice that \(\Psi(s_1)=0\) and the argument leading to \eqref{eq-key-s1} yields that 
	\[\Psi(s)\in 2\pi\Z.\]
Moreover,  \(\Psi\) is continuous,  by  Lemma~\ref{lem:ext-wn} (use \eqref{eq:cont-wn}).  Eventually,   we get that
\[\Psi(s)=0\quad(s_1<s\leq s_2),\]
and this finishes the proof of the proposition.
\end{proof}
\medskip
\begin{proof}[Proof of Lemma~\ref{lem:ext-wn}]~\medskip

\noindent{\bf Step 1.} \medskip

Let us introduce a phase function that we will use throughout the proof. 
Let \(0<\beta<\alpha\).   Pick  arbitrary \(\varsigma_1,\varsigma_2\in(s_1-\beta,s_2+\beta]\) such that \(\varsigma_1<\varsigma_2\).  By Taylor's formula,  we have
\begin{equation}\label{eq:Ch-T}
\wg(\varsigma_1)=\wg(\varsigma_2)+(\varsigma_1-\varsigma_2)\wg\,'(\varsigma_2)+\mathcal O(|\varsigma_1-\varsigma_2|^2).
\end{equation}
Since \(|\wg\,'(\varsigma_2)|=1\),   we deduce from the previous formula  that
\begin{equation}\label{eq:Ch-TT}
|\wg(\varsigma_1)-\wg(\varsigma_2)|=\varsigma_2-\varsigma_1+\mathcal O(|\varsigma_1-\varsigma_2|^2).
\end{equation}
Inserting \eqref{eq:Ch-TT} into \eqref{eq:Ch-T},  and using that all error terms  are uniform in $|\varsigma_1-\varsigma_2|\leq\beta$ due to the uniform continuity of $\gg$ and its derivatives, we get that
\[\wg(\varsigma_1)=\wg(\varsigma_2)+|\wg(\varsigma_1)-\wg(\varsigma_2)|\bigl(-\wg\,'(\varsigma_2)+\mathcal O(|\varsigma_1-\varsigma_2|)\bigr),\]
which can be rewritten in the form
\[\frac{\wg(\varsigma_1)-\wg(\varsigma_2)}{|\wg(\varsigma_1)-\wg(\varsigma_2)|}=-\wg\,'(\varsigma_2)+\mathcal O(|\varsigma_1-\varsigma_2|)\quad\mbox{for }\varsigma_1<\varsigma_2<s_2+\beta.\]
This proves that the function
\[g(\varsigma_1,\varsigma_2):=\begin{cases}
\displaystyle\frac{\wg(\varsigma_1)-\wg(\varsigma_2)}{|\wg(\varsigma_1)-\wg(\varsigma_2)|}&\mbox{if }\varsigma_1<\varsigma_2\medskip\\
-\wg\,'(\varsigma_2)&\mbox{if }\varsigma_1\geq \varsigma_2
\end{cases}\]
is continuous on \(D:=(s_1-\alpha,s_2+\alpha)^2\)  and is valued in the unit sphere \(\mathbb S^1\).    We can then find a continuous function \(\phi(\varsigma_1,\varsigma_2)\) defined on \(D\) such that \(\ee^{\ii\phi(\varsigma_1,\varsigma_2)}=g(\varsigma_1,\varsigma_2)\) on \(D\),  where we have identified \(\R^2\) and \(\C^2\) in the standard manner (see, e.g. \cite[Lem. 1.1, pp. 5-6]{BM}).

Similarly,  the function  
\[h(\varsigma_1,\varsigma_2):=\begin{cases}
\displaystyle\frac{\wg(\varsigma_1)-\wg(\varsigma_2)}{|\wg(\varsigma_1)-\wg(\varsigma_2)|}&\mbox{if }\varsigma_1>\varsigma_2\medskip\\
\wg\,'(\varsigma_1)&\mbox{if }\varsigma_1\leq \varsigma_2
\end{cases}\]
is continuous on \(D \) and there is a continuous phase function \(\psi(\varsigma_1,\varsigma_2)\) such that \(\ee^{\ii\psi(\varsigma_1,\varsigma_2)}=h(\varsigma_1,\varsigma_2)\).  In general,  we cannot find a global continuous phase  function \(\varphi\)  on \(D\) such that 
\[ \frac{\wg(\varsigma_1)-\wg(\varsigma_2)}{|\wg(\varsigma_1)-\wg(\varsigma_2)|}=\ee^{\ii \varphi(\varsigma_1,\varsigma_2)} \mbox{ for }\varsigma_1\not=\varsigma_2.\]
 \medskip

\noindent{\bf Step 2.} \medskip

We prove that  the second limit in \eqref{eq:ext-wn*} exists; the proof that the first limit exists is similar.  For any \(s\in(s_2,s_2+\alpha)\),  there exists a  continuous real-valued function 
\(\Th_{s}\) defined on \([s_1,s]\) such that
\begin{equation}\label{eq:Ch-Th}
\wg(\varsigma)=\wg(s)+|\wg(\varsigma)-\wg(s)| 
\bigl(\cos\Th_s(\varsigma),\sin\Th_s(\varsigma)\bigr)^\top\quad\mbox{for }s_1\leq \varsigma< s.
\end{equation}
In fact,  from Step~1,  we can take \(\Th_{s}(\varsigma)=\phi(\varsigma,s)\),  \(s_1\leq\varsigma\leq s\).  
Consequently the following limit exists,
\[
\lim_{s\searrow s_2}w_{\,\wG_{s_1,s_2}}\bigl(\widetilde\gamma(s)\bigr)=
\lim_{s\searrow s_2}
\bigl(\phi(s_2,s)-\phi(s_1,s)\bigr)
=\phi(s_2,s_2)-\phi(s_1,s_2).
\]

\noindent{\bf Step 3.} \medskip

We prove \eqref{eq:cont-wn}.  Fix \(s_0\in[s_1,s_2]\).  By Steps 1 and 2,   
\[  w_{\,\wG_{s_1,s_0}}\bigl(\widetilde\gamma(s_0)\bigr):=\lim_{s\searrow s_0}
w_{\,\wG_{s_1,s_0}}\bigl(\widetilde\gamma(s)\bigr)=\phi(s_0,s_0)-\phi(s_1,s_0),\]
and for any \(s\in (s_1,s_2+\alpha)\),  we have
\[  w_{\,\wG_{s_1,s}}\bigl(\widetilde\gamma(s)\bigr):=\lim_{\tilde s\searrow s}
w_{\,\wG_{s_1,s}}\bigl(\widetilde\gamma(\tilde s)\bigr)
=\phi(s,s)-\phi(s_1,s).\]
Then,
\[ \lim_{\substack{\\s\to s_0\\s\in(s_1,s_2+\alpha)}} w_{\,\wG_{s_1,s}}\bigl(\widetilde\gamma(s)\bigr)=\phi(s_0,s_0)-\phi(s_1,s_0)
=w_{\,\wG_{s_1,s_0}}\bigl(\widetilde\gamma(s_0)\bigr).\]
\end{proof}
\subsection{Proof of Proposition \ref{prop:geometric_bound}}
\label{ssec:proof}

\noindent{\bf Step 1: reduction to a simpler geometric setting.}
	Without loss of generality we can assume that $c_1$ coincides with the origin and that $c_2 = (a,0)^\top$ with some $a > 0$. In this notation, we have  $a = |c_1-c_2|$. The problem can be reduced to the situation when $t = 1$ and $p_1 = (0,-1)^\top$, $p_2 = (a,-1)^\top$. In order to reduce to the case $t=1$, we scale the curve $\G$ with the factor $\frac{1}{t}$. Upon such a scaling the length of the curve and the distance between the points $c_1$ and $c_2$
	are both multiplied with the factor $\frac{1}{t}$, while the total curvature remains invariant. 
	From now on we assume that $t =1$ and the inequality we need to prove reads as follows
	\begin{equation}\label{eq:ineq1}
	|\G| \ge |c_1-c_2|+\int_{s_1}^{s_2}\kp(s)\dd s.
	\end{equation}
	Furthermore, in the case $p_1\ne (0,-1)^\top$ and $p_2\ne(a,-1)^\top$ we can add to the curve $\G$ arcs of the circles $\p\cB_1(c_1)$ and $\p\cB_1(c_2)$ which connect $p_1$ with $(0,-1)^\top$ and $p_2$ with $(a,-1)^\top$, respectively\footnote{Strictly speaking, since we want $\G$ to be simple, one would have to put the additional arcs slightly inside the disks $\cB_1(c_1)$. $\cB_1(c_2)$, and use a limiting argument.}, and set these lowest points of the circles $\cB_1(c_j)$, $j=1,2$, as new $p_1$ and $p_2$; see Figure~\ref{fig-curve}.
Under such a geometric transform the length of the curve $\G$ will clearly increase by the same quantity as the total curvature of $\G$ meaning that we only need to prove~\eqref{eq:ineq1} for such a location of the points $p_1,p_2$.	 
\begin{figure}[h]
\centering
\includegraphics[width=12cm]{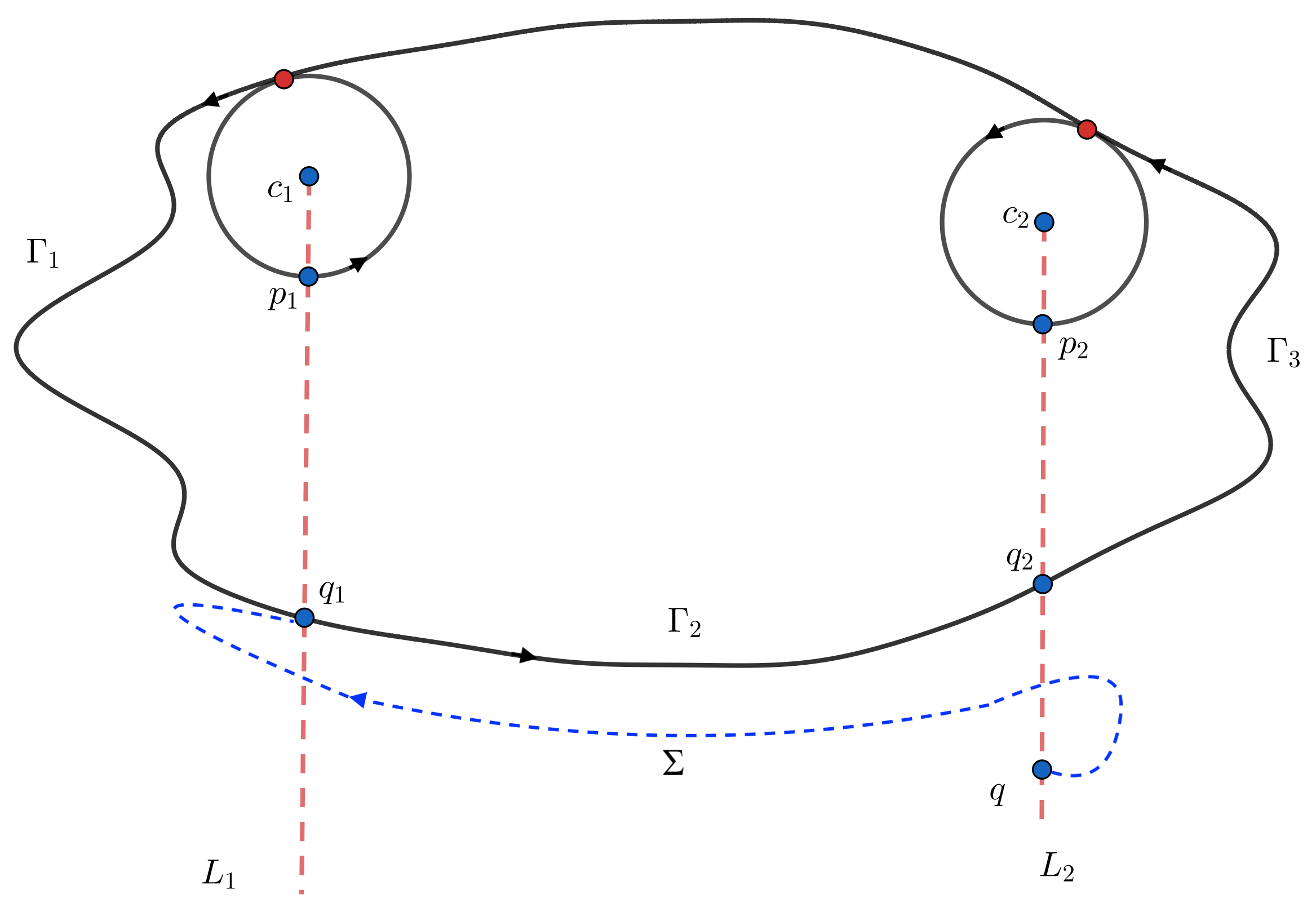}
\caption{Illustration of the reduction to a setting where $p_1+(0,-1)^\top$ and $p_2+(a,-1)^\top$, possibly after adding arcs of circles. The splitting of $\Gamma$ into $\Gamma_1,\Gamma_2$ and $\Gamma_3$ is shown, along with the auxiliary curve $\Sigma$ (dashed).}
\label{fig-curve}
\end{figure}

	\medskip
	
	\noindent {\bf Step 2: splitting the curve $\G$ into three parts.}
	We introduce the ray $L_1 := \{(0,y)\colon y\in\dR_-\}$ starting at $c_1$ and passing through $p_1$ and the ray $L_2 := \{(a,y)\colon y\in\dR_-\}$
	starting at $c_2$ and passing through $p_2$. We introduce
	$r_1\in [s_1,s_2]$ such that $\gg(r_1)\in L_1$ and that for any $r\in[s_1,s_2]$ with $\gg(r)\in L_1$ it holds that $\gg_2(r_1)\le \gg_2(r)$. Note that such a value $r_1$ exists, since the curve $\G$ intersects the ray $L_1$ at least in one point $p_1$. Furthermore, we define
	\[
	r_2 :=\inf\big\{r\in[r_1,s_2]\colon\gg(r)\in L_2\big\}.
	\]
	We introduce the points 
	\[
	q_j := \gg(r_j)\in\dR^2,\qquad j=1,2.
	\]	
	Now, we can split the curve $\G$ into three parts
	\[
	\G_{1} := \gg([s_1,r_1]),\qquad \G_{2} := \gg([r_1,r_2]),\qquad \G_{3} := \gg([r_2,s_2]), 
	\]
	see Figure~\ref{fig-curve}.
	\medskip
	
	\noindent {\bf Step 3: auxiliary curves.}
 
	We define an auxiliary smooth simple curve $\Sg\subset\dR^2$ such that
	\begin{myenum}
		\item The starting point $q\in\dR^2$ of $\Sg$ is on the ray $L_2$ and the
		terminal point of $\Sg$ is $q_1$.
		\item The \(y\)-coordinate of \(q\) is lower than the \(y\)-coordinates of both \(q_1\) and \(q_2\). 
		\item $\Sg\cup\G_2\cup\G_3$ is a smooth simple curve.
		\item The total curvature of $\Sg$ is zero.
		\item $w_{\Sg\cup\G_2}(c_2)\in \{-2\pi,0\}$. 
	\end{myenum}
	A curve $\Sg$ satisfying all the above assumptions can be constructed  from the terminal point to the starting point as follows (see Figure~\ref{fig-curve} for illustration). We proceed backwards from the point $q_1$ as a smooth continuation of the curve $\G_2$, then we return back to the ray $L_1$ below the point $q_1$, then we proceed along the ray $L_1$ until no parts of $\G$ are below in the sense of $y$-coordinate,  then we move towards the starting point $q\in L_2$ . In order to achieve zero total curvature we construct the curve $\Sg$ so that the angle between the tangent vectors to $\Sg$ at the point $q_1$ and the ray $L_1$  coincides with the angle between the tangent vector to $\Sg$ at the point $q$ and the ray $L_2$. 
	
	It is clear from the construction that $w_{\Sg\cup\G_2}(c_2) \in 2\pi\dZ$, because the starting point and the endpoint of the curve $\Sg\cup\G_2$ both lie on the ray $L_2$ emanating from the point $c_2$.  Since the curve $\Sg\cup\G_2$ (apart from a small piece at the start) hits the ray $L_2$ only at the starting point and the terminal point,  we can extend \(\Sigma\cup\Gamma_2\) to a simple closed curve by adding a piece \(\widetilde\G_2\) to the right of \(L_2\) (except possibly for a small piece near \(q\) or \(q_2\)).  Since \(\Sigma\cup \G_2\cup\widetilde\G_2\) is a simple closed curve, its winding number with respect to \(c_2\) is \(-2\pi\), \(0\) or \(2\pi\).  Since the \(y\)-coordinate of \(q\) is lower than the one of \(q_2\),  we can exclude the case \(2\pi\).  Finally,  since \(w_{\wG_2}(c_2)=0\),  we obtain by \eqref{eq:additivity} that the only possible cases are $w_{\Sg\cup\G_2}(c_2) = 0$ and $w_{\Sg\cup\G_2}(c_2) = -2\pi$.  \medskip
	
Next, we introduce two more auxiliary curves. Namely, the curve $\Sg_1\subset\dR^2$, which lies inside the disk $\cB_1(c_1)$  connecting the point $c_1$ with $p_1$ and satisfying the following properties
	\begin{myenum}
		\item[(a$_1$)] \(\Sigma_1\) is a simple curve and the total curvature of $\Sg_1$ is zero.
		\item[(b$_1$)] $w_{\Sg_1}(c_1) = -\frac{\pi}{2}$.
	\end{myenum}
Note that \(w_{\Sigma_1}(q_1)=0\) since \(c_1\) and \(p_1\) are both on \(L_1\),  but both of their \(y\)-coordinates are larger than the one of \(q_1\), which is also on \(L_1\).

	Analogously,  we construct
	the curve $\Sg_2\subset\dR^2$, which lies inside the disk $\cB_1(c_2)$  connecting the point $c_2$ with $p_2$ and satisfying the following properties
	\begin{myenum}
		\item[(a$_2$)] \(\Sigma_2\) is a simple  curve and the total curvature of $\Sg_2$ is zero.
		\item[(b$_2$)] $w_{\Sg_2}(c_2) = -\frac{\pi}{2}$.
	\end{myenum}
	The curves $\Sg$, $\Sg_1$ and $\Sg_2$ are shown in Figure~\ref{fig-curve*}.
 \begin{figure}[htp]
\centering
\includegraphics[width=12cm]{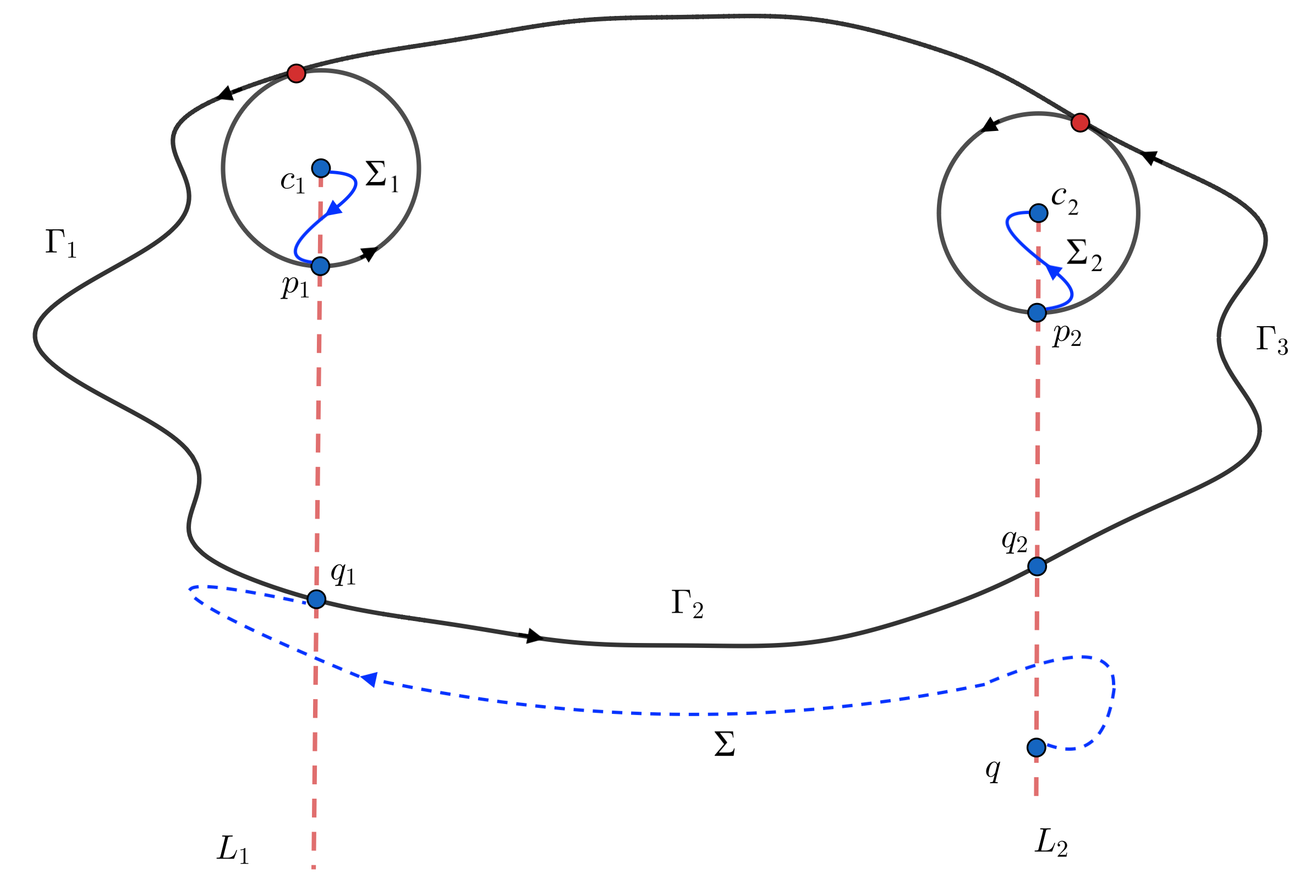}
\caption{Illustration of the curves $\Sigma$, $\Sigma_1$ and $\Sigma_2$. The curve $\Sigma$ (dashed) connects the points $q$ and $q_1$. The curve $\Sigma_1$ connects the point $p_2$ and the center $c_1$ of the disk. Similarly, the curve $\Sigma_2$ connects the center $c_2$ of the disk and the point $p_2$.}
\label{fig-curve*}
\end{figure}
	
	\medskip
	
	\noindent {\bf Step 4: estimates of the lengths.}
	Firstly,
	we notice that
	\begin{equation}\label{eq:key}
	|\G_2| \ge |c_1-c_2|, 
	\end{equation}	
	using that $\dist(L_1,L_2) = |c_1-c_2|$ and that $\G_2$ connects a point on $L_1$ with a point on $L_2$.
	
	Secondly, we infer that
	\begin{equation}\label{eq:g13wind}
	|\G_1| \ge w_{\G_1}(c_1),\qquad |\G_3| \ge w_{\G_3}(c_2).
	\end{equation}
	Let us only show the first of the above two inequalities in \eqref{eq:g13wind}. The second of them can be shown analogously. Identifying $\dR^2$ with the complex plane $\dC$ and fixing $c_1$ to be the origin, we can parametrize $\G_1$ in the complex plane by the mapping 
	\[
	[s_1,r_1]\ni s\mapsto r(s) e^{\ii\Th(s)},
	\]
	where $\Th$ is the continuous angle function as in Definition~\ref{dfn:winding}
	associated with the curve $\G_1$ and
	$r(s) :=|\gg(s)|$. Hence, we derive that
	\[
	\begin{aligned}
	|\G_1| &= \int_{s_1}^{r_1}|r'(s) + \ii \Th'(s)r(s)|\dd s \ge  
	\int_{s_1}^{r_1}|\Th'(s)|r(s)\dd s\\
	&\ge
	\int_{s_1}^{r_1}|\Th'(s)|\dd s
	\ge
	\left|\int_{s_1}^{r_1}\Th'(s)\dd s\right| = |\Th(r_1) - \Th(s_1)| = |w_{\G_1}(c_1)| \ge w_{\G_1}(c_1),
	\end{aligned}
	\]
	where we used in between that $r(s) \ge 1$ for all $s\in[s_1,r_1]$,  since the curve \(\G_1\) lies outside the circle \(\cB_1(0)\).\medskip
	
	\noindent {\bf Step 5: final estimates.}
	In this step we combine all the obtained estimates and use extensively the properties of the winding numbers stated in Subsection~\ref{sec:winding}. With a slight abuse of notation we denote by $\kp$ the curvature of any of the curves introduced above as no confusion can arise.
	First, we obtain 
	\begin{equation}\label{eq:estimate1}
	\begin{aligned}
	\int_{\G_1}\kp  = \int_{\Sg_1\cup\G_1}\kp &= w_{\Sg_1\cup\G_1}(c_1) + w_{\Sg_1\cup\G_1}(q_1)\\
	&=w_{\Sg_1}(c_1) + w_{\G_1}(c_1) + w_{\Sg_1}(q_1) + w_{\G_1}(q_1)\\
	&\le -\frac{\pi}{2} + |\G_1| + w_{\G_1}(q_1),
	\end{aligned}
	\end{equation}
	where we used that the total curvature of $\Sg_1$ is zero in the first step (see (a$_1$)), Proposition~\ref{prop:curvature_winding}
	in the second step, additivity of the winding numbers in the third step (see \eqref{eq:additivity}),  and finally employed the properties  $w_{\Sg_1}(c_1) =-\frac{\pi}{2}$ (see (a$_2$)), $w_{\Sg_1}(q_1) = 0$, and $w_{\G_1}(c_1) \le |\G_1|$ (see \eqref{eq:g13wind}) in the last step.
	Next, we obtain that
	\begin{equation}\label{eq:estimate2}
	\begin{aligned}
	\int_{\G_2\cup\G_3}\kp = \int_{\Sg\cup\G_2\cup\G_3\cup\Sg_2}\kp &= w_{\Sg\cup\G_2\cup\G_3\cup\Sg_2}(q)+
	w_{\Sg\cup\G_2\cup\G_3\cup\Sg_2}(c_2)\\
	&= w_{\Sg\cup\G_2\cup\G_3\cup\Sg_2}(q) + 
	w_{\Sg\cup\G_2}(c_2) + w_{\G_3}(c_2) + w_{\Sg_2}(c_2)\\
	&\le
	-\frac{\pi}{2}+|\G_3| +
	w_{\Sg\cup\G_2\cup\G_3\cup\Sg_2}(q),
	\end{aligned}
	\end{equation}
	where we used that the total curvatures of $\Sg$ and $\Sg_2$ are both equal to zero in the first step, applied Proposition~\ref{prop:curvature_winding} in the second step, used additivity of the winding numbers in the third step (see \eqref{eq:additivity}),  and combined the properties $w_{\Sg_2}(c_2) = -\frac{\pi}{2}$ (see (b$_2$)), $w_{\Sg\cup\G_2}(c_2) \in\{-2\pi,0\}\le0$ (see (v)), and $w_{\G_3}(c_2)\le |\G_3|$ (see \eqref{eq:g13wind}) in the last step.
	It remains to notice that by construction, one has\footnote{In fact, this is a consequence of the invariance of the winding number  by rotation and translation (see Figure~\ref{fig-id}).}
 \[
	w_{\Sg\cup\G_2\cup\G_3\cup\Sg_2}(q)
	+w_{\G_1}(q_1)\in\{-\pi,\pi\}.
	\]  
\begin{figure}[h]
\centering
\includegraphics[width=12cm]{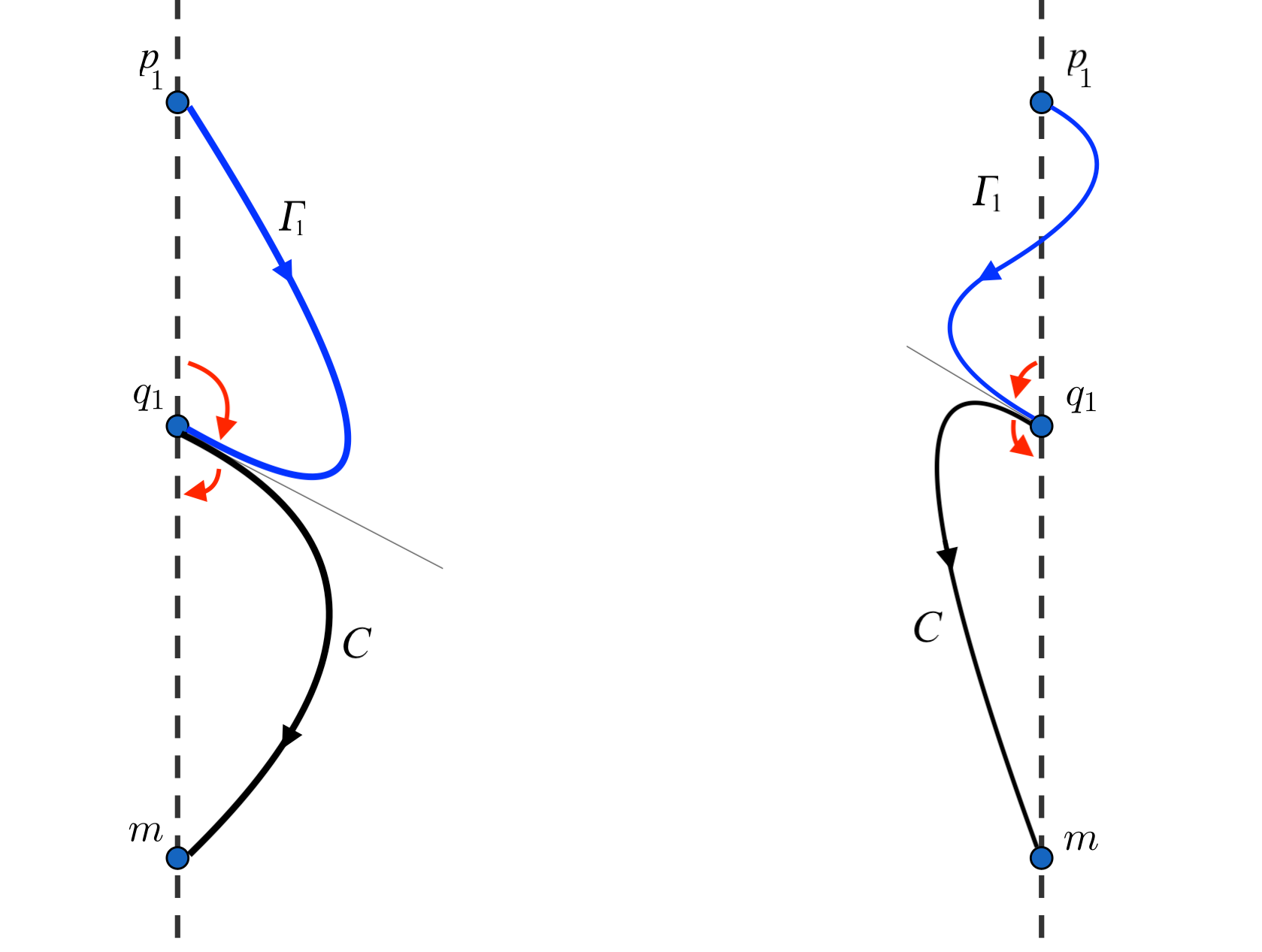}
\caption{A schematic illustration of the calculation of $w_{\Sg\cup\G_2\cup\G_3\cup\Sg_2}(q)
	+w_{\G_1}(q_1)$. The curve $\Sigma\cup\Gamma_2\cup\Gamma_3\cup\Sigma_2$ with start point $q$ and terminal point $c_2$ is rotated about $q$, then translated along $\overrightarrow{qq_1}$. The resulting curve, denoted  by $C$, starts at $q_1$  and terminates at $m$,  where $m$ is the point obtained from  $c_2$ after rotation and translation. Moreover $C$ is tangent to $\G_1$ at $q_1$. Finally one observes that $w_{\G_1}(q_1)+w_{C}(q_1)=\pm \pi$.}
 \label{fig-id}
\end{figure}

\noindent Combining estimates~\eqref{eq:estimate1} and~\eqref{eq:estimate2} we end up with
	\[
	\int_\G\kp  = \int_{\G_1}\kp + \int_{\G_2\cup\G_3}\kp\le -\pi + |\G_1|+|\G_3| + 	w_{\Sg\cup\G_2\cup\G_3\cup\Sg_2}(q)
	+w_{\G_1}(q_1) \le |\G_1|+|\G_3|.
	\]
	Finally, using that $|c_1-c_2|\le |\G_2|$ we get
	\[
	|c_1-c_2|+\int_\G\kp \le |\G_1|+|\G_2|+|\G_3| = |\G|,
	\]
	which is the desired inequality.

\end{document}